\documentclass[a4paper,10pt,leqno]{amsart}
\usepackage{amsmath,amssymb,amsfonts,amsthm}
\usepackage[top=30mm,bottom=30mm,left=25mm,right=25mm]{geometry}
\usepackage{setspace}
\usepackage{cite}
\usepackage{color}
\usepackage[pagewise]{lineno}
\usepackage[dvipdfmx]{graphicx}

\theoremstyle{plain}
\newtheorem{theorem}{Theorem}[section]
\newtheorem{proposition}[theorem]{Proposition}
\newtheorem{lemma}[theorem]{Lemma}
\newtheorem{remark}[theorem]{Remark}
\newtheorem{corollary}[theorem]{Corollary}

\newtheorem{problem}[theorem]{Problem}

\theoremstyle{definition}
\newtheorem{definition}[theorem]{Definition}

\theoremstyle{plain}
	\newtheorem{maintheorem}{Theorem}

\makeatletter
  
  \@addtoreset{equation}{section}
\makeatother

\newcommand{\RN}{{\mathbb{R}^N}}
\newcommand{\R}{\mathbb{R}}
\newcommand{\e}{\varepsilon}
\newcommand{\calL}{\mathcal{L}}

\def\bu{\bar{u}}

\def\calF{\mathcal{F}}
\def\calJ{\mathcal{J}}
\def\tf{\tilde{f}}
\def\tJ{\tilde{J}}
\def\hf{\hat{f}}
\def\hF{\hat{F}}

\begin{document}

\title[Semilinear heat equations]{
Thresholds on growth of nonlinearities and\\
singularity of initial functions\\
for semilinear heat equations}

\author{Yasuhito Miyamoto}
\thanks{The first author was supported by JSPS KAKENHI Grant Numbers 19H01797, 19H05599.}
\address{Graduate School of Mathematical Sciences, The University of Tokyo, 3-8-1 Komaba, Meguro-ku, Tokyo 153-8914, Japan}
\email{miyamoto@ms.u-tokyo.ac.jp}

\author{Masamitsu Suzuki}
\thanks{The second author was supported by Grant-in-Aid for JSPS Fellows No. 20J11985.}
\address{Graduate School of Mathematical Sciences, The University of Tokyo, 3-8-1 Komaba, Meguro-ku, Tokyo 153-8914, Japan}
\email{masamitu@ms.u-tokyo.ac.jp}

\begin{abstract}
Let $N\ge 1$ and let $f\in C[0,\infty)$ be a nonnegative nondecreasing function and $u_0$ be a possibly singular nonnegative initial function.
We are concerned with existence and nonexistence of a local in time nonnegative solution in a uniformly local Lebesgue space of a semilinear heat equation
\[
\begin{cases}
\partial_tu=\Delta u+f(u) & \textrm{in}\ \RN\times(0,T),\\
u(x,0)=u_0(x) & \textrm{in}\ \RN
\end{cases}
\]
under mild assumptions on $f$.
A relationship between a growth of $f$ and an integrability of $u_0$ is studied in detail.
Our existence theorem gives a sharp integrability condition on $u_0$ in a critical and subcritical cases, and it can be applied to a regularly or rapidly varying function $f$.
In a doubly critical case existence and nonexistence of a nonnegative solution can be determined by special treatment.
When $f(u)=u^{1+2/N}[\log(u+e)]^{\beta}$, a complete classification of existence and nonexistence of a nonnegative solution is obtained.
We also show that the same characterization as in Laister {\it et.\ al.\ }[11] is still valid in the closure of the space of bounded uniformly continuous functions in the space $L^r_{\rm ul}(\RN)$.
Main technical tools are a monotone iterative method, $L^p$-$L^q$ estimates, Jensen's inequality and differential inequalities.
\end{abstract}

\date{\today}
\maketitle

\date{\today}
\section{Introduction and main results}
We are concerned with existence and nonexistence of a local in time solution for a semilinear heat equation
\begin{equation}\label{S1E1}
\begin{cases}
\partial_tu=\Delta u+f(u) & \textrm{in}\ \RN\times (0,T),\\
u(x,0)=u_0(x) & \textrm{in}\ \RN,
\end{cases}
\end{equation}
where the domain is $\RN$, $N\ge 1$, $f$ is a $C^1$ function and the initial function $u_0$ may be unbounded.
When $u_0\in L^{\infty}(\RN)$, it is known that a solution can be constructed by contraction mapping theorem.
On the other hand, when $u_0\not\in L^{\infty}(\RN)$, the existence of a solution is not trivial, and it depends on the balance between a growth of $f$ and a strength of singularities of $u_0$, {\it i.e.}, an integrability of $u_0$. 
Weissler~\cite{W80} studied the power case $f(u)=|u|^{p-1}u$ and obtained the following:
\begin{proposition}\label{S1P1}
Let $f(u)=|u|^{p-1}u$, $p>1$ and $r_c:=N(p-1)/2$.
Then the following hold:
\begin{enumerate}
\item (Existence) The problem (\ref{S1E1}) admits a local in time solution $u(t)\in C([0,T),L^r(\RN))$ if one of the following holds:
\begin{enumerate}
\item (Subcritical case) $r>r_c$, $r\ge 1$ and $u_0\in L^r(\RN)$.
\item (Critical case) $r=r_c>1$ and $u_0\in L^r(\RN)$.
\end{enumerate}
\item (Nonexistence) For each $1\le r<r_c$, there exists a nonnegative function $u_0\in L^r(\RN)$ such that (\ref{S1E1}) admits no nonnegative solution.
\end{enumerate}
\end{proposition}
Let $u(x,t)$ be a solution of $\partial_t u=\Delta u+|u|^{p-1}u$.
Let $\lambda>0$ and $u_{\lambda}(x,t)=\lambda^{2/(p-1)}u(\lambda x,\lambda^2t)$.
Then, $u_{\lambda}$ also satisfies the same equation.
We see that $\left\|u_{\lambda}(x,0)\right\|_{r}=\left\|u(x,0)\right\|_{r}$ if and only if $r=r_c$.
Proposition~\ref{S1P1} shows that
\[
u_0\in L^{r_c}(\RN)
\]
is an optimal integrability condition for the solvability.
For the case $f(u)=|u|^{p-1}u$, much attention has been paid and a brief history can be found in \cite{LRSV16}.
See also \cite{CW98,W81,QS07} for various results.

As mentioned in \cite{LS20}, a tight correspondence between $f$ and the integrability of $u_0$ fails in the case where $f(u)\neq |u|^{p-1}u$.
Then, two problems arise:
\begin{enumerate}
\item[(A)] given $f$, characterize the set $S$ of initial data for which (\ref{S1E1}) has a solution;
\item[(B)] given the set $S$ of initial data, characterize the nonlinearity $f$ for which (\ref{S1E1}) has a solution for every initial data in $S$.
\end{enumerate}
With regards to (B), Laister {\it et.~al.~}\cite{LRSV16} gave a complete answer.
In \cite{LRSV16} the following was proved:
Let $f$ be a nonnegative nondecreasing continuous function and $\Omega$ be a smooth bounded domain.
Then, a Cauchy Dirichlet problem
\begin{equation}\label{S1E1+}
\begin{cases}
\partial_t u=\Delta u+f(u) & \textrm{in}\ \Omega\times (0,T),\\
u=0 & \textrm{on}\ \partial\Omega\times (0,T),\\
u(x,0)=u_0(x) & \textrm{in}\ \Omega
\end{cases}
\end{equation}
admits a local in time nonnegative solution for every nonnegative initial data $u_0\in L^r(\Omega)$ if and only if
\begin{equation}\label{LRSV}
\begin{cases}
\limsup_{u\to\infty}\frac{f(u)}{u^{1+2r/N}}<\infty & \textrm{if}\ 1<r<\infty,\\
\int_1^{\infty}\frac{\tf(u)}{u^{1+{2}/{N}}}du<\infty & \textrm{if}\ r=1,
\end{cases}
\end{equation}
where $\tf(u)=\sup_{1\le \tau\le u}\frac{f(\tau)}{\tau}$.
In the case $r=1$ various properties were studied in \cite{LS20}.

In this paper we mainly study Problem (A) and also study Problem (B) under a general integrability condition on $u_0$.
We prepare some notation.
Let $1\le r\le \infty$.
We define uniformly local $L^r$ spaces by
\[
L^r_{\rm ul}(\RN):=\left\{ u\in L^1_{\rm loc}(\RN)\left|\ \left\|u\right\|_{L^r_{\rm ul}(\RN)}<\infty\right.\right\}.
\]
Here, for $\rho>0$, $B(y,\rho):=\{x\in\RN|\ |x-y|<\rho\}$ and
\[
\left\|u\right\|_{L^r_{\rm ul}(\RN)}:=
\begin{cases}
\sup_{y\in\RN}\left(\int_{B(y,1)}|u(x)|^rdx\right)^{1/r} & \textrm{if}\ \ 1\le r<\infty,\\
{\rm{esssup}}_{y\in\RN}\left\| u\right\|_{L^{\infty}(B(y,1))} & \textrm{if}\ \ r=\infty.
\end{cases}
\]
We easily see that $L^{\infty}_{\rm ul}(\RN)=L^{\infty}(\RN)$ and that $L^{\beta}_{\rm ul}(\RN)\subset L^{\alpha}_{\rm ul}(\RN)$ if $1\le \alpha\le\beta$.
We define $\calL^r_{\rm ul}(\RN)$ by
\[
\calL^r_{\rm ul}(\RN):=
\overline{BUC(\RN)}^{\|\,\cdot\,\|_{L^r_{\rm ul}(\RN)}},
\]
{\it i.e.}, $\calL^r_{\rm ul}(\RN)$ denotes the closure of the space of bounded uniformly continuous functions $BUC(\RN)$ in the space $L^r_{\rm ul}(\RN)$.
We assume
\begin{equation}\label{f}
f\in C^1(0,\infty)\cap C[0,\infty),\ f(u)>0\ \textrm{for}\ u>0,\ f'(u)\ge 0\ \textrm{for}\ u>0,\ \ F(u)<\infty\ \textrm{for}\ u>0,\tag{f}
\end{equation}
where
\[
F(u):=\int_u^{\infty}\frac{d\tau}{f(\tau)}.
\]
We define $X_q$ by
\[
X_q:=\left\{
f\in C[0,\infty)\left|\ \textrm{$f$ satisfies (\ref{f}) and the limit $q:=\lim_{u\to\infty}f'(u)F(u)$ exists.}\right.
\right\}.
\]
In \cite{FI18,M18} it was proved that if the limit $q$ exists, then $q\in[1,\infty]$.
Let us explain the exponent $q$.
If $f\in C^2$, then by L'Hospital's rule we have
\[
q=\lim_{u\to\infty}\frac{F(u)}{1/f'(u)}=\lim_{u\to\infty}\frac{(F(u))'}{(1/f'(u))'}=\lim_{u\to\infty}\frac{f'(u)^2}{f(u)f''(u)}.
\]
The growth rate of $f$ can be defined by $p:=\lim_{u\to\infty}uf'(u)/f(u)$.
We  apply L'Hospital's rule. Then,
\[
\frac{1}{p}=\lim_{u\to\infty}\frac{(f(u)/f'(u))'}{(u)'}
=\lim_{u\to\infty}\left(1-\frac{f(u)f''(u)}{f'(u)^2}\right)=1-\frac{1}{q},
\ \textrm{and hence}\ \frac{1}{p}+\frac{1}{q}=1.
\]
The $q$ exponent is the conjugate exponent of the growth rate $p$.
For example, if $f(u)=u^p$ ($p>1$), then $q=p/(p-1)$.
The leading term is not necessarily a pure power function $u^p$.
If $f(u)=u^p[\log (u+e)]^{\beta}$ $(p>1,\ \beta\in\R)$, then $q=p/(p-1)$.
The case $q=1$ corresponds to the superpower case.
For instance, the $q$ exponent becomes $1$ if 
\[
f(u)=\exp(u^{p})\ (p>0),\ \ f(u)=\exp(\underbrace{\cdots\exp(u)\cdots}_{n\ \textrm{times}})
\ \ \textrm{or}\ \ f(u)=\exp(|\log u|^{p-1}\log u)\ (p>1).
\]
Fujishima-Ioku~\cite{FI18} studied Problem (A) for $f\in X_q$ and obtained the following:
\begin{proposition}\label{S1P2}
The following hold:
\begin{enumerate}
\item (Existence) Let $u_0\ge 0$. Suppose that $f\in X_q$ and
\begin{equation}\label{S1P2E1}
f'(u)F(u)\le q\ \ \textrm{for large $u>0$.}
\end{equation}
Then (\ref{S1E1}) has a local in time nonnegative solution if one of the following holds:
\begin{enumerate}
\item (Subcritical case) $r>N/2$, $q\le 1+r$ and $F(u_0)^{-r}\in L^1_{\rm ul}(\RN)$.
\item (Critical case) $r=N/2$, $q<1+r$ and $F(u_0)^{-r}\in\calL^1_{\rm ul}(\RN)$.
\end{enumerate}
\item (Nonexistence) Suppose that $f\in C^2[0,\infty)\cap X_q$ with $q<1+N/2$ and that $f''(u)\ge 0$ for $u\ge 0$.
If $0<r<N/2$ and $q\le 1+r$, then there exists a nonnegative initial function $u_0$ such that $F(u_0)^{-r}\in L^1_{\rm ul}(\RN)$ and (\ref{S1E1}) admits no nonnegative solution.
\end{enumerate}
\end{proposition}

\noindent
\begin{remark}\label{S1R0}
\mbox{}
\begin{enumerate}
\item In Proposition~\ref{S1P2}~(ii) we can take $u_0\in\mathcal{L}^1_{\rm ul}(\RN)$.
See the proof of \cite[Theorem~1.2]{FI18} for details.
\item Proposition~\ref{S1P2} shows that, for each $f\in X_q$ with (\ref{S1P2E1}), an optimal integrability condition is
\[
F(u_0)^{-N/2}\in\calL^1_{\rm ul}(\RN).
\]
When $f(u)=u^p$, then $F(u)^{-N/2}=(p-1)^{N/2}u^{N(p-1)/2}$ for $u>0$.
Therefore, the case $r=N/2$ is a critical case which corresponds to Proposition~\ref{S1P1}~(i)~(b).
\item Let $q_{S}:=(N+2)/4$ be the conjugate exponent of the critical Sobolev exponent $(N+2)/(N-2)$.
In \cite{M18} a radial singular stationary solution $u^*(x)$ of (\ref{S1E1}) near the origin was constructed if $f\in X_q$ with $q<q_{S}$.
Moreover, $u^*$ is unique under a certain assumption on $f$ (see \cite{MN18,MN20}) and
\[
u^*(x)=F^{-1}\left(\frac{|x|^2}{2N-4q}(1+o(1))\right)
\ \ \textrm{as}\ \ |x|\to 0.
\]
Since $F(u^*)^{-r}\not\in L^1_{\rm ul}(\RN)$ for $r\ge N/2$ and $F(u^*)^{-r}\in L^1_{\rm ul}(\RN)$ for $r<N/2$, $u^*$ is on a border between Proposition~\ref{S1P2} (i) and (ii).
\end{enumerate}
\end{remark}

Let
\[
S(t)[\phi](x):=\int_{\RN}K(x,y,t)\phi(y)dy\ \ \textrm{for}\ \ \phi\in L^1_{\textrm{ul}}(\RN),
\]
where $K(x,y,t):=(4\pi t)^{-N/2}\exp\left(-{|x-y|^2}/{4t}\right)$.
Then, $S(t)\phi$ gives a solution of the Cauchy problem of the linear heat equation $\partial_tu=\Delta u$ with the initial function $\phi(x)$.
We define a solution of (\ref{S1E1}).
\begin{definition}\label{S1D1}
We call $u(t)$ a solution of (\ref{S1E1}) if there exists $T>0$ such that $u(t)\in L^{\infty}((0,T),L^1_{\rm ul}(\RN)) \cap L_{\rm loc}^{\infty}((0,T),L^{\infty}(\RN))$ and $u$ satisfies
\begin{equation}\label{S1D1E1}
\infty>u(t)=\calF[u(t)]\ \ \textrm{for a.e.\ $x\in\RN$, $0<t<T$,}
\end{equation}
where
\[
\calF[u(t)]:=S(t)u_0+\int_0^tS(t-s)f(u(s))ds.
\]
We call a measurable finite almost everywhere function $\bu:\ \RN\times(0,T)\to\R$ a supersolution for (\ref{S1E1}) if there exists $T>0$ such that $\bu(t)\ge\calF[\bu(t)]$ for a.e. $x\in\RN$, $0<t<T$.
\end{definition}
The first result is a generalization of Proposition~\ref{S1P2}~(i).
Specifically, the technical assumption (\ref{S1P2E1}) can be removed as the following (i) and (iii) show.
\begin{maintheorem}\label{THA}
Let $u_0\ge0$ and $f\in X_q$.
Then (\ref{S1E1}) has a local in time nonnegative solution $u(t)$, $0<t<T$, if one of the following holds:
\begin{enumerate}
\item (Subcritical case 1) $r>N/2$, $q<1+r$  and $F(u_0)^{-r}\in L^1_{\rm ul}(\RN)$.
\item (Subcritical case 2) $r>N/2$, $q=1+r$, $F(u_0)^{-r}\in L^1_{\rm ul}(\RN)$ and (\ref{S1P2E1}) holds.
\item (Critical case) $r=N/2$, $q<1+r$ and $F(u_0)^{-r}\in \calL^1_{\rm ul}(\RN)$.
\end{enumerate}
Moreover, in all cases, there exists $C>0$ such that
\begin{equation}\label{THAE0}
\left\|F(u(t))^{-r}\right\|_{L^1_{\rm ul}(\RN)}\le C\ \ \textrm{for}\ 0<t<T.
\end{equation}
\end{maintheorem}

\begin{remark}\label{S1R1}
\mbox{}
\begin{enumerate}
\item Theorem~\ref{THA}~(i) with $q>1$ was proved in \cite[Theorem 1.4]{GM21} and Theorem~\ref{THA}~(ii) is included in Proposition~\ref{S1P2}~(i).
However, in this paper we prove three cases in a unified way, using a different approach.
See the proof of Theorem~\ref{TH1} for details.
\item If $q<1+r$ or if $q=1+r$ with (\ref{S1P2E1}), then $F(u)^{-r}$ is convex for large $u>0$.
Therefore, $F(u_0)^{-r}\in L^1_{\rm ul}(\RN)$ always implies $u_0\in L^1_{\rm ul}(\RN)$ and $S(t)u_0$ is well defined.
\item If $q= 1+r$ and (\ref{S1P2E1}) does not hold, then $F(u)^{-r}$ may be nonconvex in $u$, and $F(u_0)^{-r}\in L^1_{\rm ul}(\RN)$ does not necessarily imply $u_0\in L^1_{\rm ul}(\RN)$.
Hence the case $q=1+r$ is critical in some sense.
\item Using the method used in the proof of \cite[Theorem~1]{W86}, we see that Proposition~\ref{S1P2}~(ii) also holds even if we adopt Definition~\ref{S1D1}.
Theorem~\ref{THA} and Proposition~\ref{S1P2}~(ii) complete a classification of the existence and nonexistence problem for $f\in X_q$ in a reasonable region $\{(q,r)|\ 1\le q\le 1+r,\ r>0\}$ except $(q,r)=(1+N/2,N/2)$.
\end{enumerate}
\end{remark}
\bigskip

{
\begin{figure}[tbp]\label{fig1}
\centering \leavevmode
{\unitlength 0.1in%
\begin{picture}(29.5000,22.4500)(1.3000,-23.0000)%
%
\special{pn 8}%
\special{pa 290 2010}%
\special{pa 3080 2010}%
\special{fp}%
\special{sh 1}%
\special{pa 3080 2010}%
\special{pa 3013 1990}%
\special{pa 3027 2010}%
\special{pa 3013 2030}%
\special{pa 3080 2010}%
\special{fp}%
%
\special{pn 8}%
\special{pa 600 2300}%
\special{pa 600 100}%
\special{fp}%
\special{sh 1}%
\special{pa 600 100}%
\special{pa 580 167}%
\special{pa 600 153}%
\special{pa 620 167}%
\special{pa 600 100}%
\special{fp}%
%
\special{pn 8}%
\special{pa 2800 210}%
\special{pa 710 2300}%
\special{fp}%
%
\special{pn 8}%
\special{pa 1000 170}%
\special{pa 1000 2010}%
\special{fp}%
%
\special{pn 20}%
\special{pa 1010 1200}%
\special{pa 1810 1200}%
\special{fp}%
%
\special{pn 8}%
\special{pa 1810 1200}%
\special{pa 1810 2010}%
\special{dt 0.045}%
%
\special{pn 8}%
\special{pa 600 1200}%
\special{pa 1000 1200}%
\special{dt 0.045}%
%
\special{pn 8}%
\special{pa 1000 220}%
\special{pa 2800 210}%
\special{ip}%
%
\special{pn 4}%
\special{pa 1890 1110}%
\special{pa 1010 230}%
\special{fp}%
\special{pa 1920 1080}%
\special{pa 1060 220}%
\special{fp}%
\special{pa 1950 1050}%
\special{pa 1120 220}%
\special{fp}%
\special{pa 1980 1020}%
\special{pa 1180 220}%
\special{fp}%
\special{pa 2010 990}%
\special{pa 1240 220}%
\special{fp}%
\special{pa 2040 960}%
\special{pa 1300 220}%
\special{fp}%
\special{pa 2070 930}%
\special{pa 1360 220}%
\special{fp}%
\special{pa 2100 900}%
\special{pa 1420 220}%
\special{fp}%
\special{pa 2130 870}%
\special{pa 1480 220}%
\special{fp}%
\special{pa 2160 840}%
\special{pa 1540 220}%
\special{fp}%
\special{pa 2190 810}%
\special{pa 1600 220}%
\special{fp}%
\special{pa 2220 780}%
\special{pa 1660 220}%
\special{fp}%
\special{pa 2250 750}%
\special{pa 1720 220}%
\special{fp}%
\special{pa 2280 720}%
\special{pa 1780 220}%
\special{fp}%
\special{pa 2310 690}%
\special{pa 1840 220}%
\special{fp}%
\special{pa 2340 660}%
\special{pa 1900 220}%
\special{fp}%
\special{pa 2370 630}%
\special{pa 1950 210}%
\special{fp}%
\special{pa 2400 600}%
\special{pa 2010 210}%
\special{fp}%
\special{pa 2430 570}%
\special{pa 2070 210}%
\special{fp}%
\special{pa 2460 540}%
\special{pa 2130 210}%
\special{fp}%
\special{pa 2490 510}%
\special{pa 2190 210}%
\special{fp}%
\special{pa 2520 480}%
\special{pa 2250 210}%
\special{fp}%
\special{pa 2550 450}%
\special{pa 2310 210}%
\special{fp}%
\special{pa 2580 420}%
\special{pa 2370 210}%
\special{fp}%
\special{pa 2610 390}%
\special{pa 2430 210}%
\special{fp}%
\special{pa 2640 360}%
\special{pa 2490 210}%
\special{fp}%
\special{pa 2670 330}%
\special{pa 2550 210}%
\special{fp}%
\special{pa 2700 300}%
\special{pa 2610 210}%
\special{fp}%
\special{pa 2730 270}%
\special{pa 2670 210}%
\special{fp}%
\special{pa 2760 240}%
\special{pa 2730 210}%
\special{fp}%
\special{pa 1860 1140}%
\special{pa 1000 280}%
\special{fp}%
\special{pa 1830 1170}%
\special{pa 1000 340}%
\special{fp}%
\special{pa 1790 1190}%
\special{pa 1000 400}%
\special{fp}%
\special{pa 1730 1190}%
\special{pa 1000 460}%
\special{fp}%
\special{pa 1670 1190}%
\special{pa 1000 520}%
\special{fp}%
\special{pa 1610 1190}%
\special{pa 1000 580}%
\special{fp}%
\special{pa 1550 1190}%
\special{pa 1000 640}%
\special{fp}%
\special{pa 1490 1190}%
\special{pa 1000 700}%
\special{fp}%
\special{pa 1430 1190}%
\special{pa 1000 760}%
\special{fp}%
\special{pa 1370 1190}%
\special{pa 1000 820}%
\special{fp}%
\special{pa 1310 1190}%
\special{pa 1000 880}%
\special{fp}%
\special{pa 1250 1190}%
\special{pa 1000 940}%
\special{fp}%
\special{pa 1190 1190}%
\special{pa 1000 1000}%
\special{fp}%
\special{pa 1130 1190}%
\special{pa 1000 1060}%
\special{fp}%
\special{pa 1070 1190}%
\special{pa 1000 1120}%
\special{fp}%
%
\special{pn 4}%
\special{pa 1560 1200}%
\special{pa 1000 1760}%
\special{fp}%
\special{pa 1680 1200}%
\special{pa 1000 1880}%
\special{fp}%
\special{pa 1440 1200}%
\special{pa 1000 1640}%
\special{fp}%
\special{pa 1320 1200}%
\special{pa 1000 1520}%
\special{fp}%
\special{pa 1200 1200}%
\special{pa 1000 1400}%
\special{fp}%
\special{pa 1080 1200}%
\special{pa 1000 1280}%
\special{fp}%
%
\special{pn 4}%
\special{sh 1}%
\special{ar 1810 1200 16 16 0 6.2831853}%
\put(10.3000,-21.5000){\makebox(0,0){1}}%
\put(4.7000,-21.4000){\makebox(0,0){$O$}}%
\put(18.3000,-21.6000){\makebox(0,0){1+$\frac{N}{2}$}}%
\put(4.8000,-12.2000){\makebox(0,0){$\frac{N}{2}$}}%
\put(4.7000,-1.2000){\makebox(0,0){$r$}}%
\put(30.6000,-21.4000){\makebox(0,0){$q$}}%
\put(27.4000,-7.3000){\makebox(0,0){$r=q-1$}}%
%
\special{pn 8}%
\special{pa 2000 1320}%
\special{pa 1810 1200}%
\special{fp}%
\special{sh 1}%
\special{pa 1810 1200}%
\special{pa 1856 1253}%
\special{pa 1855 1228}%
\special{pa 1877 1219}%
\special{pa 1810 1200}%
\special{fp}%
\put(19.9000,-14.1000){\makebox(0,0)[lb]{(doubly critical)}}%
\put(10.7000,-17.1000){\makebox(0,0)[lb]{{\colorbox[named]{White}{nonexistence}}}}%
\put(11.4000,-8.6000){\makebox(0,0)[lb]{{\colorbox[named]{White}{existence}}}}%
\put(11.3000,-5.1000){\makebox(0,0)[lb]{{\colorbox[named]{White}{existence(subcritical)}}}}%
\put(11.4000,-10.6000){\makebox(0,0)[lb]{{\colorbox[named]{White}{(critical)}}}}%
%
\special{pn 8}%
\special{pa 1410 1050}%
\special{pa 1410 1190}%
\special{fp}%
\special{sh 1}%
\special{pa 1410 1190}%
\special{pa 1430 1123}%
\special{pa 1410 1137}%
\special{pa 1390 1123}%
\special{pa 1410 1190}%
\special{fp}%
\end{picture}}%
\caption{Theorem~\ref{THA} is for the existence region.
Proposition~\ref{S1P2}~(ii) is for the nonexistence region.
Theorems~\ref{THB} and \ref{THC} are for a doubly critical case.
Theorem~\ref{THD} is an example of a doubly critical case, {\it i.e.,} $f(u)=u^{1+2/N}[\log(u+e)]^{\beta}$.
}
\end{figure}
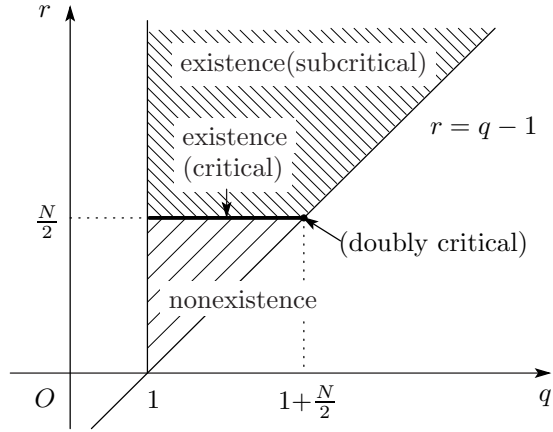
}

Let us consider the case where $(q,r)=(1+N/2,N/2)$.
This case corresponds to the case $r=r_c=1$ in Proposition~\ref{S1P1} and it is not covered by Propositions~\ref{S1P1}, \ref{S1P2} or Theorem~\ref{THA}.
The simplest example is $f(u)=u^{1+2/N}$.
Then, the integrability condition becomes $F(u_0)^{-r}=(2/N)^{N/2}u_0\in\mathcal{L}^1_{\rm ul}(\RN)$.
It is known that there exists a nonnegative initial data $u_0\in L^1(\RN)(\subset\mathcal{L}^1_{\rm ul}(\RN))$ such that (\ref{S1E1}) admits no nonnegative solution.
See \cite{BC96,CZ03,LRSV16,LS20,W86} for nonexistence results.
This case is quite delicate and referred as a doubly critical case in \cite[Section 7.5]{BC96}, since $r=N/2$ and $q=1+r$.
See Figure~\ref{fig1}.
A sufficient condition for existence is recently studied in \cite{M21}.
Combining a nonexistence result \cite{BP85} and an existence result \cite{M21}, we have the following:
\begin{proposition}\label{S1P2+}
Let $f(u)=|u|^{2/N}u$ and 
\[
Z_r:=\left\{
\phi(x)\in L^1_{\rm loc}(\RN)\left|
\int_{\RN}|\phi|\left[\log(|\phi|+e)\right]^rdx<\infty\right.
\right\}.
\]
Then the following hold:
\begin{enumerate}
\item If $u_0\in Z_r$ for some $r\ge N/2$, then (\ref{S1E1}) admits a local in time solution.
\item For each $0\le r<N/2$, there is a nonnegative initial function $u_0\in Z_r$ such that (\ref{S1E1}) admits no nonnegative solution.
\end{enumerate}
\end{proposition}
Since $Z_r\subset Z_{N/2}$ for $r\ge N/2$, by Proposition~\ref{S1P2+} we see that
\[
u_0\in Z_{N/2}
\]
is an optimal integrability condition for the solvability when $f(u)=|u|^{2/N}u$.
In particular, $Z_{N/2}$ is a proper subset of $L^1(\RN)$.

We study existence of a solution in a doubly critical case when $f$ is a general nonlinearity.
The next main theorem is a generalization of Proposition~\ref{S1P2+}~(i).
\begin{maintheorem}[Existence, doubly critical case]\label{THB}
Let $u_0\ge0$ and $q=1+N/2$.
Suppose that $f\in X_q$ holds.
Let
\begin{multline}\label{THBE0}
g(u):=u[\log(u+e)]^{\alpha},\ \ h(u):=F(u)^{-N/2}\\
\textrm{and}\ \
J_{\alpha}(u):=g(h(u))=F(u)^{-N/2}\left[\log\left(F(u)^{-N/2}+e\right)\right]^{\alpha}.
\end{multline}
Then (\ref{S1E1}) admits a local in time nonnegative solution $u(t)$, $0<t<T$, if one of the following holds:
\begin{enumerate}
\item There exists $\alpha>N/2$ such that $J_{\alpha}(u_0)\in L^1_{\rm ul}(\RN)$ and $J_{\alpha}(u)$ is convex for large $u>0$, {\it i.e.}, $J''_{\alpha}(u)\ge 0$ for large $u>0$.
\item $J_{\alpha}(u_0)\in\calL^1_{\rm ul}(\RN)$ for $\alpha=N/2$ and there exists $\rho<1$ such that
\begin{equation}\label{TH3E1}
f'(u)F(u)-q\le\frac{N\alpha}{2}\cdot\frac{\rho}{\log\left(F(u)^{-N/2}+e\right)}\ \ \textrm{for large}\ u>0.
\end{equation}
\end{enumerate}
Moreover, in two cases, there exists $C>0$ such that
\begin{equation}\label{TH3E1+}
\left\|J_{\alpha}(u(t))\right\|_{L^1_{\rm ul}(\RN)}\le
C
\ \ \textrm{for}\ 0<t<T.
\end{equation}
\end{maintheorem}
In Corollary~\ref{S5C1} it will be shown that we cannot take $\rho=1$ in Theorem~\ref{THB}~(ii).
Thus, the condition $\rho<1$ is optimal.

Let $q=1+N/2$ and $\alpha\ge N/2$.
We can easily check that if $f\in X_q$ and (\ref{S1P2E1}) hold, then $J_{\alpha}(u)$ is convex for large $u>0$ and (\ref{TH3E1}) holds for $\rho=0$.
Therefore, Theorem~\ref{THB} immediately leads to the following simple sufficient condition:
\smallskip

\noindent
{\bf Corollary B'}
{\it
Let $u_0\ge 0$ and $q=1+N/2$.
Suppose that $f\in X_q$ and (\ref{S1P2E1}) hold.
Then (\ref{S1E1}) admits a local in time nonnegative solution if $J_{\alpha}(u_0)\in L^1_{\rm ul}(\RN)$ for some $\alpha>N/2$ or $J_{\alpha}(u_0)\in\calL^1_{\rm ul}(\RN)$ for $\alpha=N/2$.
}
\bigskip

We study nonexistence of a solution in a doubly critical case.
For $\beta\in\R$, we define
\[
f_{\beta}(u):=u^{1+2/N}\left[\log(u+e)\right]^{\beta},\ \ F_\beta(u):=\int_u^\infty 
\frac{d\tau}{f_\beta(\tau)}\ \ \textrm{and}\ \ h_\beta(u):=F_\beta(u)^{-N/2}.
\]

\begin{maintheorem}[Nonexistence, doubly critical case]\label{THC}
Let $u_0\ge 0$.
Suppose that $f$ satisfies (\ref{f}) and 
there exist $C_1>0$, $C_2>0$, $\beta>0$ and $0<\delta<1$ such that
the following hold:
\begin{enumerate}
\item $F^{-1}\circ F_\beta$ is convex on $[C_1,\infty)$, {\it i.e.}, 
$f'(F^{-1}(v))\ge f'_\beta(F^{-1}_\beta(v))$ for $0<v\le F_\beta(C_1)$.
\item $F(u)\le C_2 u^{-2/N}\left[\log(u+e)\right]^{\delta}$ for $u\ge C_1$.
\end{enumerate}
Let $J_{\alpha}(u):=F(u)^{-N/2}\left[\log\left(F(u)^{-N/2}+e\right)\right]^{\alpha}$.
For each $\alpha\in [0,N/2)$, there exists a nonnegative function $u_0\in L^1_{\rm ul}(\RN)$ satisfying $J_{\alpha}(u_0)\in L^1_{\rm ul}(\RN)$ such that, for every $T>0$, (\ref{S1E1}) admits no nonnegative solution.
\end{maintheorem}

\begin{remark}\label{S1R3}
\mbox{}
\begin{enumerate}
\item If $f$ satisfies (\ref{f}) and $f'(u)F(u)\ge1+N/2$ for large $u>0$, then for each $\beta>0$, the assumption of Theorem~\ref{THC}~(i) holds by taking $C_1>0$ sufficiently large.
\item A characterization of $f$ for existence and nonexistence of a solution in $\mathcal{L}^r_{\rm ul}(\RN)$, $r\ge 1$, is given in Corollary~\ref{S9C1} and Theorem~\ref{S9T3}.
This characterization is the same one as (\ref{LRSV}) which was obtained in \cite[Corollary~4.5 and Theorem~3.4]{LRSV16}.
\end{enumerate}
\end{remark}
We study Problem (B) in Corollaries~\ref{S3C1}, \ref{S3C2}, \ref{S4C1} and \ref{S4C2}.
These corollaries give existence and nonexistence conditions on $f$ when integrability conditions on $u_0$ are given.
These corollaries are not optimal, and could be improved.
A threshold growth and a threshold integrability can be summarized as Table~\ref{tab1}.

\begin{table}[t]\label{tab1}
\caption{Relationship between a threshold growth and a threshold integrability. Here, $g_{\frac{N}{2}}(u)=u[\log(u+e)]^{N/2}$, $q=\lim_{u\to\infty}f'(u)F(u)$ and $q_J=\lim_{u\to\infty}J'(u)^2/J(u)J''(u)$.}
\begin{tabular}{|c|c|ccc|c|c|}
\hline
& problem & growth & & integrability & existence & nonexistence\\
\hline
\hline
$1\le q<1+\frac{N}{2}$ & (A) & $f(u)$ & $\!\!\rightarrow\!\!$ & $F(u_0)^{-\frac{N}{2}}\in\mathcal{L}^1_{\rm ul}(\RN)$ & Thm~\ref{THA} & Prop~\ref{S1P2}(ii)\rule[-2mm]{0mm}{6mm}\\
$1\le q_J<\infty$ & (B) & $\frac{J(u)^{1+\frac{2}{N}}}{J'(u)}$ & $\!\!\leftarrow\!\!$ & $J(u_0)\in\mathcal{L}^1_{\rm ul}(\RN)$ & Cor~\ref{S3C1} & Cor \ref{S4C1}\rule[-3mm]{0mm}{0mm}\\
\hline
 $q=1+\frac{N}{2}$ & (A) & $f(u)$ & $\!\!\rightarrow\!\!$ & $\!\! g_{\frac{N}{2}}(F(u_0)^{-\frac{N}{2}})\in\mathcal{L}^1_{\rm ul}(\RN)\!\!$ & Thm~\ref{THB} & Thm~\ref{THC} \rule[-2mm]{0mm}{6mm}\\
$q_J=\infty$ & (B) & $\!\!\frac{J(u)^{1+\frac{2}{N}}}{J'(u)\log (J(u)+e)}\!\!$ & $\!\!\leftarrow\!\!$ & $J(u_0)\in\mathcal{L}^1_{\rm ul}(\RN)$ & Cor~\ref{S3C2} & Cor~\ref{S4C2}\rule[-3mm]{0mm}{0mm}\\ 
 \hline 
\end{tabular}
\end{table}

\smallskip
We consider an example of a doubly critical case. Let $f(u)=f_{\beta}(u)$.
An elementary calculation shows that if $\beta\ge -(1+2/N)\kappa$, then $f_{\beta}(u)$ is nondecreasing for $u>0$.
Here, $\kappa$ is the largest positive root of
\begin{equation}\label{kappa}
\log\kappa+2=\kappa,\ \ \textrm{where}\ \ \kappa\simeq 3.146.
\end{equation}

The following theorem is a complete classification of integrability conditions on $u_0$.
\begin{maintheorem}[Classification for $f_\beta$]\label{THD}
Let $u_0\ge 0$, 
\[
f(u)=f_{\beta}(u)\ \ \textrm{and}\ \ J_{\alpha}(u)=F_\beta(u)^{-N/2}\left[\log\left(F_\beta(u)^{-N/2}+e\right)\right]^{\alpha}.
\]
Then the following hold:
\begin{enumerate}
\item (Existence) (\ref{S1E1}) with initial function $u_0\in L^1_{\rm ul}(\RN)$ admits a local in time nonnegative solution if one of the following holds:
\begin{enumerate}
\item $\alpha>N/2$, $\beta\ge-1$ and ${J_{\alpha}}(u_0)\in L^1_{\rm ul}(\RN)$.
\item $\alpha=N/2$, $\beta>-1$ and ${J_{\alpha}}(u_0)\in\calL^1_{\rm ul}(\RN)$.
\item $-(1+2/N)\kappa\le \beta<-1$, where $\kappa$ is given by (\ref{kappa}).
\end{enumerate}
\item (Nonexistence)
\begin{enumerate}
\item Let $\beta>-1$.
For each $\alpha\in [0,N/2)$, there exists a nonnegative function $u_0\in L^1_{\rm ul}(\RN)$ such that $J_{\alpha}(u_0)\in L^1_{\rm ul}(\RN)$ and that (\ref{S1E1}) admits no nonnegative solution.
\item Let $\beta=-1$. 
For each $\alpha\in [0,N/2]$, there exists a nonnegative function $u_0\in\calL^1_{\rm ul}(\RN)$ such that $J_{\alpha}(u_0)\in\calL^1_{\rm ul}(\RN)$ and that (\ref{S1E1}) admits no nonnegative solution.
\end{enumerate}
\end{enumerate}
\end{maintheorem}

\begin{remark}\label{S1R4}
\mbox{}
\begin{enumerate}
\item In the case $\beta>-1$ (\ref{S1E1}) is solvable for all $u_0\ge 0$ satisfying $J_{N/2}(u_0)\in \mathcal{L}^1_{\rm ul}(\RN)$.
However, in the case $\beta=-1$ (\ref{S1E1}) is not necessarily solvable even if $J_{N/2}(u_0)\in \mathcal{L}^1_{\rm ul}(\RN)$.
\item Theorem~\ref{THD} indicates that if $\beta\ge -1$, then a threshold integrability condition is $J_{N/2}(u_0)\in\mathcal{L}^1_{\rm ul}(\RN)$.
\item If $\beta=-1$, then there are $C_2>C_1>0$ such that $C_1<J'_{N/2}(u)<C_2$ for $u\ge 0$.
Therefore, $J_{N/2}(u_0)\in\mathcal{L}^1_{\rm ul}(\RN)$ if and only if $u_0\in\mathcal{L}^1_{\rm ul}(\RN)$.
\item In \cite[Section~4.4]{LRSV16} it was proved that (\ref{S1E1+}) with $f(u)=f_\beta(u)$ on a smooth bounded domain $\Omega$ is always solvable (resp.\ is not always solvable) for a nonnegative function $u_0\in L^1(\Omega)$ if $\beta<-1$ (resp.\ if $-1\le\beta\le 0$).
\end{enumerate}
\end{remark}
\bigskip

We characterize the class of nonlinearities $X_q$, since Theorems~\ref{THA} and \ref{THB} assume $f\in X_q$.
\begin{definition}
\mbox{}
\begin{enumerate}
\item Let ${\rm RV}_p$, $0\le p<\infty$, denote the set of {\it regularly varying functions}, {\it i.e.}, 
$f\in {\rm RV}_p$ if
$$\lim_{u\to\infty}\frac{f(\lambda u)}{f(u)}=\lambda^p.$$
In particular, if $f\in {\rm RV}_0$, then $f$ is called a {\it slowly varying function}.
\item Let ${\rm RV}_{\infty}$ denote the set of {\it rapidly varying functions}, i.e,
$f\in {\rm RV}_{\infty}$ if
\[
\lim_{u\to\infty}\frac{f(\lambda u)}{f(u)}=
\begin{cases}
\infty & \textrm{for}\ \lambda>1,\\
0 & \textrm{for}\ 0<\lambda<1.
\end{cases}
\]
\end{enumerate}
\end{definition}
The class ${\rm RV}_p$ is a generalization of a homogeneous function of degree $p$ and ${\rm RV}_{\infty}$ is a generalization of a superpower function, {\it e.g.}, $e^u$.
Readers can consult the book \cite{GD87} for details of ${\rm RV}_p$.
\begin{maintheorem}\label{THE}
Assume that $p$ and $q$ satisfy the following: $p:=q/(q-1)$ if $q>1$, and $p:=\infty$ if $q=1$.
Then the following hold:
\begin{enumerate}
\item If $f\in X_q$ for some $q\in [1,\infty)$, then $f\in {\rm RV}_p$.
\item Suppose that $f$ satisfies (\ref{f}) and that $f'$ is nondecreasing.
Then, $f\in X_q$ for $q\in (1,\infty)$ if and only if $f\in{\rm RV}_p$ for $p\in (1,\infty)$.
\item Suppose that $f$ satisfies (\ref{f}) and that $f'(u)F(u)$ is nondecreasing.
Then, $f\in X_1$ if and only if $f\in{\rm RV}_{\infty}$.
\end{enumerate}
\end{maintheorem}
Theorem~\ref{THE}~(ii) and (iii) say that $X_q$ ($1\le q<\infty$) and ${\rm RV}_p$ ($1<p\le\infty$) are equivalent.
Therefore, Theorems~\ref{THA} and \ref{THB} can be applied to $f\in {\rm RV}_p$ under additional assumptions.
It follows from Karamata's representation theorem, which is stated in Proposition~\ref{S7P1}, that for each function $f\in {\rm RV}_p$, $1<p<\infty$, $f(u)$ has a concrete form (\ref{S7P1E1}) which explicitly describes a function of $X_q$.
Moreover, it is known that $f\in {\rm RV}_p$, $0<p<\infty$, can be written as $f(u)=u^pL(u)$ for $u>1$, where $L\in {\rm RV}_0$, {\it i.e.}, a slowly varying function.

\bigskip
Let us explain technical details.
In the existence part a critical case (Theorem~\ref{THA}~(iii) without (\ref{S1P2E1})) or a doubly critical case (Theorem~\ref{THB}) were not covered by existing results.
Since these cases are delicate, we introduce a new method.
First, we separately treat the nonlinear term $f$ and a convex function $J$, which appears in an integrability condition $J(u_0)\in L^1_{\rm ul}(\RN)$.
We introduce a simple but new supersolution (\ref{TH1E3}), using $J$.
In \cite{IKS16,GM21,RS13} similar functions were also used as supersolutions.
However, these supersolutions were directly related to integrability conditions.
In Theorem~\ref{TH1} we show that (\ref{TH1E3}) is actually a supersolution for (\ref{S1E1}), and hence by monotone iterative method we can construct a nonnegative solution.
In \cite{FI18} a change of variables was used to construct a supersolution, and (\ref{S1P2E1}) was necessary.
Essential conditions for $J$ are (\ref{TH1E1}) and (\ref{TH1E2}).
Second, we relate $f$ and $J$.
Specifically, we take $J(u)=F(u)^{-r}$ in Theorem~\ref{THA} and 
$J(u)=J_{\alpha}(u)$ in Theorem~\ref{THB}.
This method can analyze in detail a relationship between the growth of $f$ and the integrability of $u_0$, and can treat superpower nonlinearities.
Parabolic systems with superpower nonlinearities were studied in \cite{IKS16,MS19,S19}.
Theorem~\ref{TH1} is also useful in the study of Problem (B).
Using Theorem~\ref{TH1}, we give a necessary and sufficient condition on $f$ for a solvability in $\mathcal{L}^r_{\rm ul}(\RN)$, $r\ge 1$, in Section 9.
Theorem~\ref{TH1} is used in the proof of the sufficient part.
Main technical tools in the proof of Theorem~\ref{TH1} are a monotone iterative method (Proposition~\ref{S2P1}), $L^p$-$L^q$ estimates (Proposition~\ref{S2P2}) and Jensen's inequality (Proposition~\ref{S2P4}).

It is not easy to obtain a nonexistence result in a doubly critical case.
In \cite{BP85,HI18} necessary conditions on $u_0$ were obtained for $f(u)=u^p$, and nonexistence results were established.
In Theorem~\ref{S4T1} we prove a nonexistence theorem for $f_{\beta}(u)=u^{1+2/N}[\log(u+e)]^{\beta}$, $\beta>0$, which needs a more detailed analysis than previous studies.
The function $f_{\beta}$ is not homogeneous, and the function $H(t)$ defined by (\ref{eq22}), which is related to a local $L^1$-norm of a solution, is a key in the proof of Theorem~\ref{S4T1}.
The behavior of $H(t)$ gives a necessary condition for the existence of a nonnegative solution.
If we take (\ref{S4E1}) as an initial function, then we obtain a contradiction, and a nonexistence theorem for $f_{\beta}$ is proved.
The proof of Theorem~\ref{THC} is by contradiction.
Suppose that (\ref{S1E1}) has a nonnegative solution.
Using a change of variables, we can construct a supersolution for (\ref{S1E1}) with $f_{\beta}$ from  a solution of (\ref{S1E1}).
Then, it follows from a monotone iterative method that (\ref{S1E1}) with $f_{\beta}$ has a nonnegative solution.
However, (\ref{S1E1}) with $f_{\beta}$ does not have a nonnegative solution, because of Theorem~\ref{S4T1}.
Therefore, the contradiction concludes the proof of Theorem~\ref{THC}.
Main technical tools in the proofs of Theorems~\ref{S4T1} and \ref{THC} are the differential inequality (\ref{eq22+}) and a comparison principle.
Theorems~\ref{THB} and \ref{THC} are used in the proof of a complete classification for $f_{\beta}(u)$ (Theorem~\ref{THD}).

In this paper Problem (B) is also studied.
Specifically, we obtain growth conditions on $f$ for existence and nonexistence results when the integrability condition $J(u_0)\in {L}^1_{\rm ul}(\RN)$ is given.
Corollaries~\ref{S3C1} and \ref{S3C2} are derived from Theorem~\ref{TH1}.
Corollaries~\ref{S4C1} and \ref{S4C2} are counterparts of Proposition~\ref{S1P2}~(ii) and Theorem~\ref{THC}, respectively.

In \cite{LRSV16} a complete characterization for existence and nonexistence of a solution of (\ref{S1E1}) in $L^r(\Omega)$ was obtained.
Their definition of a solution is different from Definition~\ref{S1D1}, and requires that 
$u\in L^{\infty}((0,T), L^r (\Omega))$.
Only in Section~9 we adopt a similar definition of \cite{LRSV16} which is different from Definition~\ref{S1D1}, and obtain the same characterization in the $\mathcal{L}^r_{\rm ul}(\RN)$ framework.

This paper consists of ten sections.
In Section~2 several examples to which Theorem~\ref{THA} can be applied are given.
We recall basic propositions and prove useful lemmas.
They will be used in the proof of our Theorems~\ref{THA}, \ref{THB}, \ref{THC} and \ref{THD}.
In Section~3 we prove an abstract existence theorem (Theorem~\ref{TH1}) and prove Theorem~\ref{THA}.
Moreover, existence conditions on $f$ are obtained in Corollaries~\ref{S3C1} and \ref{S3C2}.
In Section~4 we prove Theorem~\ref{THC}.
Nonexistence conditions on $f$ are obtained in Corollaries~\ref{S4C1} and \ref{S4C2}.
In Section~5 we study a necessary and sufficient condition for a solvability of (\ref{S1E1}) in $\mathcal{L}^1_{\rm ul}(\RN)$.
Section~6 is devoted to the proof of Theorem~\ref{THD}.
In Section~7 we prove Theorem~\ref{THE}.
Section 8 is a summary and problems.
Sections 9 and 10 are appendices to \cite{LRSV16} and \cite{GM21}, respectively.

\section{Examples and preliminaries}
We give four examples and recall known results which are useful in the proof of the main theorems.
\subsection{Example 1. $f(u)=\exp(u^p)$, $p>0$}
By direct calculation we have
\[
q:=\lim_{u\to\infty}f'(u)F(u)=\lim_{u\to\infty}\frac{f'(u)^2}{f(u)f''(u)}
=\lim_{u\to\infty}\frac{p}{p+(p-1)u^{-p}}=1.
\]
We have
\[
\frac{d}{du}\left(\frac{f'(u)}{f(u)^{1/q}}\right)
=p(p-1)u^{p-2}.
\]
We consider the case $p\ge 1$
Since $p\ge 1$, $f'(u)/f(u)^{1/q}$ is nondecreasing.
Since
\begin{equation}\label{S2E1E1}
\frac{f'(u)}{f(u)^{1/q}}f(u)^{1/q}\int_u^{\infty}\frac{ds}{f(s)}\le f(u)^{1/q}\int_u^{\infty}\frac{f'(s)ds}{f(s)^{1/q+1}}
=f(u)^{1/q}\left[-qf(s)^{-1/q}\right]_u^{\infty}=q,
\end{equation}
we see that $f'(u)F(u)\le q$.
Proposition~\ref{S1P2}~(i) and (ii) are applicable.
Next, we consider the case $0<p<1$.
Since $f'(u)/f(u)^{1/q}$ is decreasing, by calculation similar to (\ref{S2E1E1}) we see that $f'(u)F(u)>1$.
Proposition~\ref{S1P2}~(i) is not applicable, while Theorem~\ref{THA}~(i) and (iii) are applicable.
Using Theorem~\ref{THA}~(i) and (iii) and Proposition~\ref{S1P2}~(ii), we obtain the following:
\begin{theorem}\label{S2T2}
Let $u_0\ge 0$ and $f(u)=\exp(u^p)$ $(p>0)$. Then the following hold:\\
(i) The problem (\ref{S1E1}) admits a local in time nonnegative solution if $F(u_0)^{-r}\in L^1_{\rm ul}(\RN)$ for some $r>N/2$ or $F(u_0)^{-N/2}\in \calL^1_{\rm ul}(\RN)$.\\
(ii) For each $r\in (0,N/2)$, there exists a nonnegative initial function $u_0\in L^1_{\rm ul}(\RN)$ such that $F(u_0)^{-r}\in L^1_{\rm ul}(\RN)$ and (\ref{S1E1}) admits no nonnegative solution.
\end{theorem}
By L'Hospital's rule we see that
\[
\lim_{u\to\infty}\frac{F(u)}{p^{-1}(u+1)^{-p+1}e^{-u^p}}=1.
\]
Therefore, $F(u_0)^{-r}\in L^1_{\rm ul}(\RN)$ if and only if $(u+1)^{(p-1)r}e^{ru^p}\in L^1_{\rm ul}(\RN)$.

\subsection{Example 2. $f(u)=\exp(|\log u|^{p-1}\log u)$, $p>1$}
By direct calculation we have
\[
q:=\lim_{u\to\infty}f'(u)F(u)=\lim_{u\to\infty}\frac{f'(u)^2}{f(u)f''(u)}
=\lim_{u\to\infty}\frac{1}{1+\frac{p-1}{p(\log u)^p}-\frac{1}{p(\log u)^{p-1}}}=1.
\]
We have
\[
\frac{d}{du}\left(\frac{f'(u)}{f(u)^{1/q}}\right)
=p\frac{(\log u)^{p-2}}{u^2}\left\{(p-1)-\log u\right\}<0
\ \ \textrm{for large}\ u>0.
\]
Since $f'(u)/f(u)^{1/q}$ is decreasing for large $u>0$, by calculation similar to (\ref{S2E1E1}) we see that $f'(u)F(u)> 1$.
Proposition~\ref{S1P2}~(i) is not applicable, while Theorem~\ref{THA}~(i) and (iii) are applicable.
Using Theorem~\ref{THA}~(i) and (iii) and Proposition~\ref{S1P2}~(ii), we obtain the following:
\begin{theorem}
Let $u_0\ge 0$ and $f(u)=\exp(|\log u|^{p-1}\log u)$, $p>1$.
Then the same statements as Theorem~\ref{S2T2} hold.
\end{theorem}
By L'Hospital's rule we see that
\[
\lim_{u\to\infty}\frac{F(u)}{\frac{u+e}{p[\log (u+e)]^{p-1}}e^{-|\log u|^{p-1}\log u}}=1.
\]
Therefore, $F(u_0)^{-r}\in L^1_{\rm ul}(\RN)$ if and only if
\[
\frac{[\log (u+e)]^{(p-1)r}}{(u+e)^r}e^{r|\log u|^{p-1}\log u}\in L^1_{\rm ul}(\RN).
\]
\subsection{Example 3. $f(u)=(u+a)^p/\{(p-1)\log(u+a)-1\}$, $p>1+2/N$}
We define $a:=e^{2/(p-1)}$ so that $(p-1)\log (u+a)-1\ge 1$ for $u\ge 0$.
Let $q:=p/(p-1)$.
By direct calculation we have
\[
F(u)=\frac{\log (u+a)}{(u+a)^{p-1}},
\]
hence,
\[
f'(u)F(u)=\frac{p}{p-1}+\frac{(p-1)(2p-1)\log(u+a)-p}{(p-1)\{(p-1)\log (u+a)-1\}^2}
\to q\ \ \textrm{as}\ \ u\to\infty.
\]
Since $f'(u)F(u)>q$, Proposition~\ref{S1P2} (i) is not applicable.
The statements of Theorem~\ref{THA}~(i) and (iii) and Proposition~\ref{S1P2}~(ii) hold.
Here,
\[
F(u)^{-r}=\frac{(u+a)^{(p-1)r}}{[\log(u+a)]^r}.
\]

\subsection{Example 4. The $n$-th iterated exponential function} 
Let $f(u):=\exp(\underbrace{\cdots\exp(u)\cdots}_{n\ {\rm times}})$, $n\ge 1$.
It is easy to show that $q=1$ and $f'(u)^2/(f(u)f''(u))\le 1$. See \cite{M18} for details.
Integrating $1/f(u)\le f''(u)/f'(u)^2$ over $[u,\infty)$, we have $f'(u)F(u)\le 1$.
Using Proposition~\ref{S1P2}, we obtain the following:
\begin{theorem}
Let $N\ge 1$, $u_0\ge 0$ and
$f(u):=\exp(\underbrace{\cdots\exp(u)\cdots}_{n\ {\rm times}})$, $n\ge 1$.
Then the same statements as Theorem~\ref{S2T2} hold.
\end{theorem}

\subsection{Preliminaries}
For any set $X$ and the mappings $a=a(x)$ and $b=b(x)$ from $X$ to $[0,\infty)$, we say 
\[
a(x)\lesssim b(x) \ \ \textrm{for all}\ x\in X
\]
if there exists a positive constant $C$ such that $a(x) \le Cb(x)$ for all $x\in X$.

We recall a monotone iterative method.
\begin{proposition}\label{S2P1}
Let $0<T\le\infty$ and let $f$ be a continuous nondecreasing function such that $f(0)\ge 0$.
The problem (\ref{S1E1}) has a nonnegative solution for $0<t<T$ if and only if (\ref{S1E1}) has a nonnegative supersolution $\bu(t)\in L^{\infty}((0,T),L^1_{\rm ul}(\RN))\cap L_{\rm loc}^{\infty}((0,T),L^{\infty}(\RN))$.
Moreover, if a nonnegative supersolution $\bu(t)$ exists, then the solution $u(t)$ obtained satisfies $0\le u(t)\le\bu(t)$.
\end{proposition}
We show the proof for readers' convenience.
See {\it e.g.} \cite[Theorem~2.1]{RS13} for details.
\begin{proof}
If (\ref{S1E1}) has a nonnegative solution, then the solution is also a supersolution.
Thus, it is enough to show that (\ref{S1E1}) has a nonnegative solution if (\ref{S1E1}) has a supersolution.
Let $\bu$ be a supersolution for $0<t<T$.
Let $u_1=S(t)u_0$.
We define $u_n$, $n=1,2,3,\ldots$, by
\[
u_n=\calF[u_{n-1}].
\]
Then we can show by induction that
\[
0\le u_1\le u_2\le\cdots\le u_n\le\cdots\le\bu <\infty\ \ \textrm{for a.e.}\ x\in\RN,\ 0<t<T.
\]
This indicates that the limit $\lim_{n\to\infty}u_n(x,t)$ which is denoted by $u(x,t)$ exists for almost all $x\in\RN$ and $0<t<T$.
By the monotone convergence theorem we see that
\[
\lim_{n\to\infty}\calF[u_{n-1}]=\calF(u),
\]
and hence $u=\calF(u)$.
It is clear that $0\le u(t)\le \bu(t)$.
Since $\bu(t)\in L^{\infty}((0,T),L^1_{\rm ul}(\RN))\cap L^{\infty}_{\rm loc}((0,T),L^{\infty}(\RN))$, we see that $u(t)\in L^{\infty}((0,T),L^1_{\rm ul}(\RN))\cap L^{\infty}_{\rm loc}((0,T),L^{\infty}(\RN))$
Thus, $u$ is a solution of (\ref{S1E1}).
\end{proof}

\begin{proposition}\label{S2P2}
The following hold:
\begin{enumerate}
\item Let $N\ge 1$ and $1\le\alpha\le\beta\le\infty$.
There is $C>0$ and $t_0>0$ such that, for $\phi\in L_{\rm ul}^{\alpha}(\RN)$,
\[
\left\|S(t)\phi\right\|_{L_{\rm ul}^{\beta}(\RN)}
\le {C}{t^{-\frac{N}{2}\left(\frac{1}{\alpha}-\frac{1}{\beta}\right)}}
\left\|\phi\right\|_{L_{\rm ul}^{\alpha}(\RN)}
\ \ \textrm{for}\ \ 0<t<t_0.
\]
\item Let $N\ge 1$ and $1\le\alpha<\beta\le\infty$.
Then, for each $\phi\in \calL^{\alpha}_{\rm ul}(\RN)$ and $C_*>0$, there is $t_0=t_0(C_*,\phi)$ such that
\[
\left\|S(t)\phi\right\|_{L_{\rm ul}^{\beta}(\RN)}\le C_*t^{-\frac{N}{2}\left(\frac{1}{\alpha}-\frac{1}{\beta}\right)}
\ \ \textrm{for}\ \ 0<t<t_0.
\]
\end{enumerate}
\end{proposition}
A proof of Proposition~\ref{S2P2},  which is based on \cite[Corollary~3.1]{MT06} and \cite[Lemma~8]{BC96}, can be found in \cite[Propositions~2.4 and 2.5]{GM21}.
Note that $C_*>0$ in (ii) can be chosen arbitrary small.

\begin{proposition}\label{S2P3}
Let $1\le \alpha<\infty$.
The following are equivalent:
\begin{enumerate}
\item $\phi\in\calL^{\alpha}_{\rm ul}(\RN)$.
\item $\lim_{|y|\to 0}\left\|\phi(\,\cdot\,+y)-\phi(\,\cdot\,)\right\|_{L^{\alpha}_{\rm ul}(\RN)}=0$.
\item $\lim_{t\to 0}\left\|S(t)\phi-\phi\right\|_{L^{\alpha}_{\rm ul}(\RN)}=0$.
\end{enumerate}
\end{proposition}
Fundamental properties of $S(t)$ in $L^{\alpha}_{\rm ul}(\RN)$ were studied in \cite{MT06}.
For details of Proposition~\ref{S2P3}, see \cite[Proposition~2.2]{MT06}.

\begin{proposition}\label{S2P4}{\upshape(cf. \cite[Lemma~2.4]{FI18})}
Let $C\ge 0$. The following {\upshape(i)} and {\upshape(ii)} hold:
\begin{enumerate}
\item Suppose that $J: [C,\infty)\to[0,\infty)$ is a convex function. 
If $\phi\in L^1_{\rm ul}(\RN)$, $J(\phi)\in L^1_{\rm ul}(\RN)$ and $\phi\ge C$ in $\RN$, then
\[
J(S(t)[\phi](x))\le S(t)[J(\phi)](x) \ \ \text{in $\RN\times(0,\infty)$.}
\]
\item Suppose that $K: [C,\infty)\to[0,\infty)$ is a concave function. 
If $\phi\in L^1_{\rm ul}(\RN)$ and $\phi\ge C$ in $\RN$, then
\[
K(S(t)[\phi](x))\ge S(t)[K(\phi)](x) \ \ \text{in $\RN\times(0,\infty)$.}
\]
\end{enumerate}
\end{proposition}
Proposition~\ref{S2P4} follows from Jensen's inequality.
See \cite[Proposition~2.9]{GM21} for a proof of Proposition~\ref{S2P4}.
\smallskip

Hereafter in this section we collect useful lemmas.
\begin{lemma}
\label{S2L0}
Let $C>0$. If $u\in\mathcal{L}^{1}_{\rm ul}(\RN)$, then 
$\max\{u,C\}\in\mathcal{L}^{1}_{\rm ul}(\RN)$.
\end{lemma}

\begin{proof}
Since $u\in\mathcal{L}^{1}_{\rm ul}(\RN)$, there exists a sequence 
$\{u_n\}\subset BUC(\RN)$ such that $u_n \to u$ in $L^1_{\rm ul}(\RN)$ 
as $n\to\infty$. 
Let $\{v_n (x)\}_{n=1}^{\infty}$ be defined by $v_n (x):=\max\{u_n (x),C\}$. 
We see that $\{v_n\}\subset BUC(\RN)$ and obtain
\begin{eqnarray*}
|\max\{u,C\}-v_n |\le|u-{u_n}|\to0 \ \ \text{in $L^1_{\rm ul}(\RN)$ as $n\to\infty$.}
\end{eqnarray*}
Thus $\max\{u,C\}\in\mathcal{L}^{1}_{\rm ul}(\RN)$.
\end{proof}

\begin{lemma}\label{S2L1}
Let $q\ge 1$ and $\e>0$. If $f\in X_q$, then $F(u)\lesssim u^{-1/(q-1+\e)}$ for large $u>0$.
\end{lemma}
\begin{proof}
Since $f\in X_q$, we see that $f'(u)F(u)\le q+\e$ for large $u>0$.
By this together with $f'(u) =F''(u)/F'(u)^2$ and $F'(u)=-1/f(u)<0$ we have
\[
\frac{F''(u)}{F'(u)}\ge(q+\e)\frac{F'(u)}{F(u)},
\]
which implies that $-F'(u)\gtrsim F(u)^{q+\e}$ for large $u>0$.
Then we obtain
$F(u)\lesssim u^{-1/(q-1+\e)}$ for large $u>0$.
\end{proof}

\begin{lemma}\label{S2L2}
Let $N \ge 1$, $\beta>0$ and $h_\beta(u):=F_\beta(u)^{-N/2}$.
Put $\tilde{h}_\beta(u):=(N/4)^{N/2}u[\log(u+e)]^{-N\beta/2}$. 
Then $\tilde{h}_\beta(u)\le h^{-1}_\beta(u)$ for large $u>0$.
\end{lemma}

\begin{proof}
Let $u>0$ be sufficiently large. Then it follows that
\[
F_\beta(u)=\int_u^\infty \frac{d\tau}{\tau^{1+2/N}[\log(\tau+e)]^\beta}
\ge\frac{N}{4}\int_u^\infty \frac{\frac{2}{N}\log(\tau+e)+\frac{\beta\tau}{\tau+e}}{\tau^{1+2/N}[\log(\tau+e)]^{\beta+1}}d\tau
=\frac{N}{4}u^{-2/N}[\log(u+e)]^{-\beta}.
\]
Hence, we obtain $h_\beta(u)\le(4/N)^{N/2}u[\log(u+e)]^{N\beta/2}$. We see that
\[
h_\beta(\tilde{h}_\beta(u))\le\left(\frac{4}{N}\right)^{N/2}\tilde{h}(u)[\log(\tilde{h}(u)+e)]^{\frac{N}{2}\beta}=u[\log(u+e)]^{-\frac{N}{2}\beta}[\log(\tilde{h}(u)+e)]^{\frac{N}{2}\beta}\le u.
\]
Since $h_\beta$ is increasing, $\tilde{h}_\beta(u)\le h^{-1}_\beta(u)$ for large $u>0$.
\end{proof}

\begin{lemma}\label{S2L3}
Let $N\ge1$.
Suppose that $f$ satisfies all the assumptions of Theorem~\ref{THC}.
Let $h(u):=F(u)^{-N/2}$.
Then there exists $0<\delta'<N/2$ such that $h^{-1}(u)\lesssim u[\log(u+e)]^{\delta'}$ for large $u>0$.
\end{lemma}

\begin{proof}
Let $\delta$ be given by the assumption of Theorem~\ref{THC}.
Since $0<\delta<1$, we choose $\e>0$ such that $N\delta/2+\e<N/2$.
Put $\delta':=N\delta/2+\e>0$ and $\hat{h}(u):=C_2^{N/2}u[\log(u+e)]^{\delta'}$.
Since $h(u)\ge C_2^{-N/2}u[\log(u+e)]^{-N\delta/2}$ for $u\ge C_1$, we obtain
\begin{equation}\label{eq41}
h(\hat{h}(u))\ge C_2^{-N/2}\hat{h}(u)[\log(\hat{h}(u)+e)]^{-\frac{N}{2}\delta}
=u\left(\frac{\log(u+e)}{\log(\hat{h}(u)+e)}\right)^{\frac{N}{2}\delta}
[\log(u+e)]^{\e}\ \ \textrm{for large}\ u>0.
\end{equation}
We see that
\[
\frac{\log u}{\log\hat{h}(u)}=\frac{\log u}{\log u+\delta'\log(\log(u+e))+\log C_2^{N/2}}\to1
\ \ \text{as $u\to\infty$,}
\]
which yields
\[
\left(\frac{\log(u+e)}{\log(\hat{h}(u)+e)}\right)^{\frac{N}{2}\delta}
=\left(\frac{\log(u+e)}{\log u}
\frac{\log u}{\log\hat{h}(u)}\frac{\log\hat{h}(u)}{\log(\hat{h}(u)+e)}\right)^{\frac{N}{2}\delta}\to1
\ \ \text{as $u\to\infty$.}
\]
By this together with \eqref{eq41} we have $h(\hat{h}(u))\ge u$ for large $u>0$. Since $h$ is increasing, $h^{-1}(u)\le\hat{h}(u)$ for large $u>0$.
\end{proof}

\section{Existence}
In this section a function $J$ satisfies the following:
\begin{multline}\label{J}
J\in C^2[0,\infty),\ \lim_{u\to\infty}J(u)=\infty,\ J(u)>0\ \textrm{for}\ u>0,\ J'(u)>0\  \textrm{for}\ u>0\\
\textrm{and $J'(u)$ is nondecreasing for large $u>0$.}\tag{J}
\end{multline}
The main theorem in this section is the following:
\begin{theorem}\label{TH1}
Let $N\ge 1$ and $u_0\ge0$.
Suppose that $f\in C[0,\infty)$, $f$ is nonnegative and $f$ is nondecreasing for $u>0$
and that $J$ satisfies (\ref{J}).
Suppose that there exist $\theta\in(0,1]$ and $\xi\ge0$ such that one of the following holds:
\begin{enumerate}
\item $J(u_0)\in L^1_{\rm ul}(\RN)$ and
\begin{equation}\label{TH1E1}
\lim_{\eta\to\infty}\tJ(\eta)\int_{\eta}^{\infty}\frac{\tf(\tau)J'(\tau)d\tau}{J(\tau)^{1+2/N}}=0,
\end{equation}
\item $J(u_0)\in \calL^1_{\rm ul}(\RN)$ and
\begin{equation}\label{TH1E2}
\limsup_{\eta\to\infty}\tJ(\eta)\int_{\eta}^{\infty}\frac{\tf(\tau)J'(\tau)d\tau}{J(\tau)^{1+2/N}}<\infty,
\end{equation}
\end{enumerate}
where
\begin{equation}\label{TH1E0}
\tf(u):=\sup_{\xi\le\tau\le u}\frac{f(\tau)}{J(\tau)^{\theta}}\ \ \textrm{and}\ \ 
\tJ(u):=\sup_{\xi\le\tau\le u}\frac{J'(\tau)}{J(\tau)^{1-\theta}}.
\end{equation}
Then (\ref{S1E1}) admits a local in time nonnegative solution $u(t)$ for $0<t<T$.
Moreover, there exists $C_0>0$ such that
\begin{equation}\label{TH1E2+}
\left\|J(u(t))\right\|_{L^1_{\rm ul}(\RN)}\le C_0
\ \ \textrm{for}\ 0<t<T.
\end{equation}
\end{theorem}

\begin{proof}
Let $C_1>0$ be large such that $J(u)$ is convex for $u\ge C_1$.
Let $\sigma>0$ be a constant and $u_1(x):=\max\{u_0(x),C_1,1,\xi\}$.
We see that $J(u_1)\in L^1_{\rm ul}(\RN)$ in the case (i).
By Lemma~\ref{S2L0} we see that $J(u_1)\in \calL^1_{\rm ul}(\RN)$ in the case (ii).
Since $J$ is convex for $u\ge C_1$, in two cases (i) and (ii) $J(u_1)\in L^1_{\rm ul}(\RN)$ implies $u_1\in L^1_{\rm ul}(\RN)$, and hence it follows from Proposition~\ref{S2P2} that $\left\|S(t)u_1\right\|_{\infty}<\infty$ for $t>0$.
The first term of $\mathcal{F}$ is well defined.

We define
\begin{equation}\label{TH1E3}
\bu(t):=J^{-1}\left((1+\sigma)S(t)J(u_1)\right).
\end{equation}
We show that
\begin{equation}\label{TH1E3+}
\bu(t)\in L^{\infty}((0,T),L^1_{\rm ul}(\RN))\cap L_{\rm loc}^{\infty}((0,T),L^{\infty}(\RN))
\end{equation}
for sufficiently small $T>0$.
Since $J(u_1)\in L^1_{\rm ul}(\RN)$, it follows from
Proposition~\ref{S2P2}~(i) that
\[
\left\|\bu(t)\right\|_{L^{\infty}(\RN)}\le J^{-1}((1+\sigma)Ct^{-N/2}\left\| 
J(u_1)\right\|_{L^1_{\rm ul}(\RN)})<\infty\ \ \textrm{for small}\ t>0.
\]
Hence, $\bu(t)\in L^{\infty}_{\rm loc}((0,T),L^{\infty}(\RN))$ for small $T>0$.
Since $J^{-1}(u)$ is concave for $u\ge J(C_1)$, by (\ref{TH1E3}) and Proposition~\ref{S2P2}~(i) we have
\[
\left\|\bu(t)\right\|_{L^1_{\rm ul}(\RN)}
\lesssim\left\|S(t)J(u_1)+1\right\|_{L^1_{\rm ul}(\RN)}
\le \left\|S(t)J(u_1)\right\|_{L^1_{\rm ul}(\RN)}+\left\|1\right\|_{L^1_{\rm ul}(\RN)}
\lesssim\left\|J(u_1)\right\|_{L^1_{\rm ul}(\RN)}+1,
\]
and hence $\bu\in L^{\infty}((0,T),L^1_{\rm ul}(\RN))$.
We have proved (\ref{TH1E3+}).

By Proposition~\ref{S2P4}~(i) we have
\begin{multline}\label{TH1E5}
\bu(t)-S(t)u_0
\ge \bu(t)-S(t)u_1
\ge \bu(t)-J^{-1}\left(S(t)J(u_1)\right)\\
=J^{-1}\left((1+\sigma)S(t)J(u_1)\right)-J^{-1}\left(S(t)J(u_1)\right)
=(J^{-1})'\left((1+\rho\sigma)S(t)J(u_1)\right)\sigma S(t)J(u_1)
\end{multline}
for some $\rho=\rho(x,t)\in[0,1]$.
Here, we used the mean value theorem.
Since $J(u)$ is convex for $u\ge C_1$, $J^{-1}(u)$ is concave for $u\ge J(C_1)$.
We have
\begin{multline}\label{TH1E6}
(J^{-1})'\left((1+\rho\sigma)S(t)J(u_1)\right)\sigma S(t)J(u_1)
\ge (J^{-1})'\left((1+\sigma)S(t)J(u_1)\right)\sigma S(t)J(u_1)\\
=\frac{\sigma S(t)J(u_1)}{J'\left(J^{-1}\left((1+\sigma)S(t)J(u_1)\right)\right)}
=\frac{\sigma}{1+\sigma}\frac{J(\bu(t))}{J'(\bu(t))}.
\end{multline}
By (\ref{TH1E5}) and (\ref{TH1E6}) we have
\begin{equation}\label{TH1E7}
\bu(t)-S(t)u_0\ge\frac{\sigma}{1+\sigma}\frac{J(\bu(t))}{J'(\bu(t))}.
\end{equation}
On the other hand, let $s\in (0,t)$.
Since $0<\theta\le 1$, it follows from Proposition~\ref{S2P4}~(ii) that
\begin{multline}\label{TH1E8}
S(t-s)J(\bu(s))^{\theta}=S(t-s)\left[\left\{(1+\sigma)S(s)J(u_1)\right\}^{\theta}\right]\\
\le \left\{S(t-s)\left[(1+\sigma)S(s)J(u_1)\right]\right\}^{\theta}
\le\left\{(1+\sigma)S(t)J(u_1)\right\}^{\theta}=J(\bu(t))^{\theta}.
\end{multline}
Using (\ref{TH1E8}), we have
\begin{multline}\label{TH1E9}
\int_0^tS(t-s)f(\bu(s))ds
\le\int_0^tS(t-s)\left[\left\|\frac{f(\bu(s))}{J(\bu(s))^{\theta}}\right\|_{\infty}J(\bu(s))^{\theta}\right]ds\\
\le\int_0^t\left\|\frac{f(\bu(s))}{J(\bu(s))^{\theta}}\right\|_{\infty}S(t-s)J(\bu(s))^{\theta}ds
\le J(\bu(t))^{\theta}\int_0^t\left\|\frac{f(\bu(s))}{J(\bu(s))^{\theta}}\right\|_{\infty}ds.
\end{multline}

We prove the case (i).
We have
\begin{equation}\label{TH1E10+1}
J(\bu(t))^{\theta}\int_0^t\left\|\frac{f(\bu(s))}{J(\bu(s))^{\theta}}\right\|_{\infty}ds
\le \frac{J(\bu(t))}{J'(\bu(t))}\tJ(\left\|\bu(t)\right\|_{\infty})
\int_0^t\tf(\left\|\bu(s)\right\|_{\infty})ds.
\end{equation}
We define $\eta$ by
\begin{equation}\label{TH1E10+2}
\eta:=J^{-1}(C(1+\sigma)t^{-N/2}\left\|J(u_1)\right\|_{L^1_{\rm ul}(\RN)}).
\end{equation}
Since $\left\|\bu(t)\right\|_{\infty}\le\eta$, we have
\begin{equation}\label{TH1E10+3}
\tJ(\left\|\bu(t)\right\|_{\infty})
\int_0^t\tf(\left\|\bu(s)\right\|_{\infty})ds\le
\frac{2}{N}\{C(1+\sigma)\left\|J(u_1)\right\|_{L^1_{\rm ul}(\RN)}\}^{{2}/{N}}
\tJ(\eta)
\int_{\eta}^{\infty}\frac{\tf(\tau)J'(\tau)d\tau}{J(\tau)^{1+2/N}},
\end{equation}
where we used a change of variables $\tau:=J^{-1}(C(1+\sigma)s^{-N/2}\left\|J(u_1)\right\|_{L^1_{\rm ul}(\RN)})$.
Because of (\ref{TH1E1}), there exists a large $\eta>0$ such that
\begin{equation}\label{TH1E11}
\frac{2}{N}\{C(1+\sigma)\left\|J(u_1)\right\|_{L^1_{\rm ul}(\RN)}\}^{{2}/{N}}
\tJ(\eta)
\int_{\eta}^{\infty}\frac{\tf(\tau)J'(\tau)d\tau}{J(\tau)^{1+2/N}}
\le\frac{\sigma}{1+\sigma}.
\end{equation}
Therefore, if $\eta>0$ is large, then by (\ref{TH1E11}), (\ref{TH1E10+3}), (\ref{TH1E10+1}) and (\ref{TH1E9}) we have
\begin{equation}\label{TH1E12}
\int_0^tS(t-s)f(\bu(s))ds\le\frac{\sigma}{1+\sigma}\frac{J(\bu(t))}{J'(\bu(t))}.
\end{equation}
By (\ref{TH1E12}) and (\ref{TH1E7}) we have
\begin{equation}\label{TH1E12+}
\bu(t)-S(t)u_0\ge\int_0^tS(t-s)f(\bu(s))ds.
\end{equation}
Since $\eta$ and $t$ are related by (\ref{TH1E10+2}), we have shown that there is a small $t>0$ such that $\bu(t)$ is a supersolution which satisfies (\ref{TH1E3+}).
By Proposition~\ref{S2P1} we see that (\ref{S1E1}) has a nonnegative solution $u(t)$ and that $u(t)\le\bu(t)$ for $0<t<T$.
Since $\sigma>0$ is a constant, (\ref{TH1E2+}) follows from (\ref{TH1E3}).

We prove the case (ii).
In the case (ii) the inequality (\ref{TH1E9}) also holds.
Since $J(u_1)\in \calL^1_{\rm ul}(\RN)$, it follows from
Proposition~\ref{S2P2}~(ii) that $\left\|\bu(s)\right\|_{L^\infty(\RN)}\le J^{-1}((1+\sigma)C_*s^{-N/2}\left\|J(u_1)\right\|_{L^1_{\rm ul}(\RN)})$.
By the same calculation as (\ref{TH1E10+3}) we have (\ref{TH1E10+3}) with $C$ replaced by $C_*$.
Because of (\ref{TH1E2}) and Proposition~\ref{S2P2}~(ii) we can take $C_*>0$ such that if $0<t<t_0(C_*)$, then
\begin{equation}\label{TH1E13}
\frac{2}{N}\{C_*(1+\sigma)\left\|J(u_1)\right\|_{L^1_{\rm ul}(\RN)}\}^{2/N}
\tJ(\eta)
\int_{\eta}^{\infty}\frac{\tf(\tau)J'(\tau)}{J(\tau)^{1+2/N}}d\tau
\le\frac{\sigma}{1+\sigma}.
\end{equation}
By (\ref{TH1E13}), (\ref{TH1E10+3}) with $C_*$ and (\ref{TH1E9})
we have (\ref{TH1E12}).
By (\ref{TH1E12}) and (\ref{TH1E7}) we have (\ref{TH1E12+}).
Since $\bu(t)$ is a supersolution which satisfies (\ref{TH1E3+}), by Proposition~\ref{S2P1} we see that (\ref{S1E1}) has a nonnegative solution $u(t)$ and that $u(t)\le\bu(t)$ for $0<t<T$.
Since $\sigma>0$ is a constant, (\ref{TH1E2+}) follows from (\ref{TH1E3}).
\end{proof}

\begin{proof}[\bf Proof of Theorem~\ref{THA}]
In three cases (i) (ii) and (iii) we easily see that $f\in C[0,\infty)$, $f$ is nonnegative and $f$ is nondecreasing for $u>0$, since $f\in X_q$.
Let $J(u):=F(u)^{-r}$.
We show that $J$ satisfies (\ref{J}).
It is enough to show that $J''(u)\ge 0$ for large $u$, since other properties follow from the definition of $J(u)$.
In the cases (i) and (iii) we see that $q<1+r$, and hence
\[
J''(u)=\frac{r(r+1-f'(u)F(u))}{f(u)^2F(u)^{r+2}}\ge 0\ \ \textrm{for large}\ u.
\]
In the case (ii) we see that $r+1=q\ge f'(u)F(u)$ for large $u$, and hence $J''(u)\ge 0$ for large $u$.

Next, we show that
\begin{equation}\label{THAPE1}
\frac{d}{du}\left(\frac{f(u)}{J(u)^{\theta}}\right)\ge 0\ \ \textrm{for large $u$}\ \ \textrm{and}\ \ 
\frac{d}{du}\left(\frac{J'(u)}{J(u)^{1-\theta}}\right)\ge 0\ \ \textrm{for large $u$}.
\end{equation}
We consider the cases (i) and (iii).
Then, $1\le q<1+r$ and $r\ge N/2$, and there exists $\theta>0$ such that $\frac{q-1}{r}<\theta<\min\{1,\frac{q}{r}\}$.
Then,
\begin{equation}\label{THAPE2}
\frac{d}{du}\left(\frac{f(u)}{J(u)^{\theta}}\right)
=(f'(u)F(u)-r\theta)F(u)^{r\theta-1}>0\ \ \textrm{for large}\ u>0,
\end{equation}
since $f'(u)F(u)-r\theta\to q-r\theta>0$ as $u\to\infty$.
We have
\[
\frac{d}{du}\left(\frac{J'(u)}{J(u)^{1-\theta}}\right)
=(1-\theta)J(u)^{\theta-1}J''(u)\left(\frac{1}{1-\theta}-\frac{J'(u)^2}{J(u)J''(u)}\right).
\]
Then,
\[
\frac{J'(u)^2}{J(u)J''(u)}=\frac{r}{r+1-f'(u)F(u)}\to \frac{r}{r+1-q}\ \textrm{as}\ u\to\infty.
\]
We see that $\frac{1}{1-\theta}-\frac{r}{r+1-q}>0$, and hence $J'(u)/J(u)^{1-\theta}$ is nondecreasing.
Thus, (\ref{THAPE1}) holds in the cases (i) and (iii).
We consider the case (ii).
Then $q=1+r$ and $r>N/2$.
We take $\theta=1$.
Since $q-r=1$, (\ref{THAPE2}) holds.
Since $\theta=1$, by (\ref{S1P2E1}) we see that
\[
\frac{d}{du}\left(\frac{J'(u)}{J(u)^{1-\theta}}\right)=J''(u)=\frac{r(q-f'(u)F(u))}{f(u)^2F(u)^{r+2}}\ge 0\ \ \textrm{for large}\ u.
\]
Thus, (\ref{THAPE1}) holds in the case (ii).

We prove (i) and (ii).
Then, $r>N/2$.
We check (\ref{TH1E1}).
Because of (\ref{THAPE1}), we can take $\xi>0$ such that $f(u)/J(u)^{\theta}$ and $J'(u)/J(u)^{1-\theta}$ are nondecreasing for $u>\xi$.
If $\eta>\xi$ is large, then
\[
\tJ(\eta)\int_{\eta}^{\infty}\frac{\tf(\tau)J'(\tau)d\tau}{J(\tau)^{1+2/N}}
\le\int_{\eta}^{\infty}\frac{f(\tau)J'(\tau)^2d\tau}{J(\tau)^{2+2/N}}
=r^2\int_{\eta}^{\infty}\frac{F(\tau)^{2r/N-2}}{f(\tau)}d\tau
=\frac{r^2F(\eta)^{2r/N-1}}{2r/N-1}\to 0\ \textrm{as}\ \eta\to\infty.
\]
Since $J$ satisfies all the assumptions of Theorem~\ref{TH1}~(i), by Theorem~\ref{TH1}~(i) we see that (\ref{S1E1}) has a nonnegative solution and (\ref{THAE0}) holds.

We prove (iii).
Then, $r=N/2$.
We check (\ref{TH1E2}).
Since $q-1<r\theta$, we can choose $\e>0$ such that $q-1+\e<r\theta$.
By Lemma~\ref{S2L1} we see that
\begin{equation}\label{THAPE3}
\int_{\eta}^{\infty}F(\tau)^{r\theta}d\tau
<C\int_{\eta}^{\infty}\tau^{-r\theta/(q-1+\e)}d\tau<\infty.
\end{equation}
Integrating $f'(u)/f(u)\le (q+\e)/f(u)F(u)$, we have $f(u)\le CF(u)^{-q-\e}$ for large $u$, and hence
\begin{equation}\label{THAPE4}
f(u)F(u)^{1+r\theta}\le CF(u)^{-q-\e+1+r\theta}\to 0\ \ \textrm{as}\ \ u\to\infty.
\end{equation}
Because of (\ref{THAPE1}), we can take $\xi>0$ such that $f(u)/J(u)^{\theta}$ and $J'(u)/J(u)^{1-\theta}$ are nondecreasing for $u\ge\xi$.
If $\eta>\xi$ is large, then
\[
\tJ(\eta)\int_{\eta}^{\infty}\frac{\tf(\tau)J'(\tau)d\tau}{J(\tau)^{1+2/N}}
=\frac{J'(\eta)}{J(\eta)^{1-\theta}}
\int_{\eta}^{\infty}\frac{f(\tau)J'(\tau)d\tau}{J(\tau)^{1+\theta+2/N}}
=\frac{r^2}{f(\eta)F(\eta)^{1+r\theta}}
\int_{\eta}^{\infty}F(\tau)^{r\theta}d\tau.
\]
Because of (\ref{THAPE3}) and (\ref{THAPE4}), L'Hospital's rule is applicable in the following limit:
\[
\lim_{\eta\to\infty}\frac{\int_{\eta}^{\infty}F(\tau)^{r\theta}d\tau}{f(\eta)F(\eta)^{1+r\theta}}
=
\lim_{\eta\to\infty}\frac{\frac{d}{d\eta}
\left(\int_{\eta}^{\infty}F(\tau)^{r\theta}d\tau\right)}
{\frac{d}{d\eta}\left(f(\eta)F(\eta)^{1+r\theta}\right)}
=\lim_{\eta\to\infty}\frac{-1}{f'(\eta)F(\eta)-1-r\theta}
=\frac{1}{r\theta+1-q}>0.
\]
By Theorem~\ref{TH1}~(ii) we see that (\ref{S1E1}) has a nonnegative solution and (\ref{THAE0}) holds.
\end{proof}

\begin{proof}[\bf Proof of Theorem~\ref{THB}]
In two cases (i) and (ii) we easily see that $f\in C[0,\infty)$, $f$ is nonnegative and $f$ is nondecreasing for $u>0$, since $f\in X_q$.
Let $J(u):=J_{\alpha}(u)$ and $\theta=1$. 
In both cases (i) and (ii) we have
\[
\frac{d}{du}\left(\frac{f(u)}{J(u)^{\theta}}\right)=\frac{h(u)^{-1+2/N}}{[\log(h(u)+e)]^{\alpha}}
\left\{f'(u)F(u)-\frac{N}{2}\left(1+\frac{\alpha h(u)}{(h(u)+e)\log(h(u)+e)}
\right)\right\}\ge 0\ \ \textrm{for large}\ u,
\]
since $f'(u)F(u)\to1+N/2$ as $u\to\infty$.
In the case (i) we see that $\frac{d}{du}\left(\frac{J'(u)}{J(u)^{1-\theta}}\right)=J''(u)\ge 0$, since $J$ is convex.
In the case (ii) let $\hat{\rho}>\rho$. By direct calculation we have
\begin{equation}\label{eq01+}
\frac{d}{du}\left(g'(h(u))^{\hat{\rho}}h'(u)\right)=g'(h(u))^{\hat{\rho}-1}\cdot
\frac{N}{2f(u)^2F(u)^{N+2}}\cdot\frac{[\log(h(u)+e)]^{\alpha-1}}{h(u)+e}\cdot j(u),
\end{equation}
where
\begin{align}\label{eq02+}
\begin{split}
j(u)&:=\frac{N}{2}\alpha\hat{\rho}\left\{1+\frac{e}{h(u)+e}+\frac{(\alpha-1)h(u)}{(h(u)+e)\log(h(u)+e)}
\right\}+(q-f'(u)F(u))\left\{\left(1+\frac{e}{h(u)}\right)\log(h(u)+e)+\alpha\right\}\\
&\ge\frac{N}{2}\alpha\hat{\rho}\left\{1+\frac{e}{h(u)+e}+\frac{(\alpha-1)h(u)}{(h(u)+e)\log(h(u)+e)}
\right\}-\frac{N}{2}\alpha\rho
\left\{1+\frac{e}{h(u)}+\frac{\alpha}{\log(h(u)+e)}\right\}\\
&\to\frac{N}{2}\alpha(\hat{\rho}-\rho)>0 \ \ \textrm{as}\ u\to\infty.
\end{split}
\end{align}
Note that we use (\ref{TH3E1}).
Considering the case where $\hat{\rho}=1$, we see that $J''(u)\ge 0$ for large $u>0$.
In both cases (i) and (ii) we have checked (\ref{J}) and (\ref{THAPE1}).

We prove (i). 
Because of (\ref{THAPE1}), we can take a large $\xi>0$ such that $\tf(u)=f(u)/J(u)^{\theta}$ for $u\ge\xi$ and $\tJ(u)=J'(u)/J(u)^{1-\theta}$ for $u\ge\xi$.
Then,
\begin{multline*}
\frac{J'(\eta)}{J(\eta)^{1-\theta}}\int_{\eta}^{\infty}\frac{f(\tau)J'(\tau)d\tau}{J(\tau)^{1+\theta+2/N}}
\le\int_{\eta}^{\infty}\frac{f(\tau)J'(\tau)^2d\tau}{J(\tau)^{2+2/N}}
=\int_{\eta}^{\infty}\frac{f(\tau)g'(h(\tau))^2h'(\tau)^2d\tau}{g(h(\tau))^{2+2/N}}\\
\le\frac{N}{2}C
\int_{h(\eta)}^{\infty}\frac{2d\sigma}{(\sigma+e)[\log(\sigma+e)]^{2\alpha/N}}
=\frac{NC}{2\alpha/N-1}[\log(h(\eta)+e)]^{1-2\alpha/N}\to 0\ \ \textrm{as}\ \ \eta\to\infty,
\end{multline*}
and hence (\ref{TH1E1}) holds.
By Theorem~\ref{TH1}~(i) we see that (\ref{S1E1}) has a nonnegative solution and that (\ref{TH3E1+}) holds.

We prove (ii).
Because of (\ref{THAPE1}), we can take a large $\xi>0$ such that $\tf(u)=f(u)/J(u)^{\theta}$ for $u\ge\xi$ and $\tJ(u)=J'(u)/J(u)^{1-\theta}$ for $u\ge\xi$.
We choose $\tilde{\rho}\in(\rho,1)$.
It follows from \eqref{eq01+} and \eqref{eq02+} that
$g'(h(u))^{\tilde{\rho}}h'(u)$ is increasing for large $u>0$.
When $\eta>0$ is large, we have
\begin{multline*}
\frac{J'(\eta)}{J(\eta)^{1-\theta}}\int_{\eta}^{\infty}\frac{f(\tau)J'(\tau)d\tau}{J(\tau)^{1+\theta+2/N}}
=g'(h(\eta))^{1-\tilde{\rho}}g'(h(\eta))^{\tilde{\rho}}h'(\eta)
\int_{\eta}^{\infty}\frac{f(\tau)g'(h(\tau))h'(\tau)d\tau}{g(h(\tau))^{2+2/N}}\\
\le g'(h(\eta))^{1-\tilde{\rho}}
\int_{\eta}^{\infty}\frac{f(\tau)g'(h(\tau))^{1+\tilde{\rho}}h'(\tau)^2d\tau}{g(h(\tau))^{2+2/N}}
=\frac{N}{2}g'(h(\eta))^{1-\tilde{\rho}}\int_{h(\eta)}^{\infty}\frac{\sigma^{1+2/N}g'(\sigma)^{1+\tilde{\rho}}d\sigma}{g(\sigma)^{2+2/N}}\\
\le \frac{N}{2}C[\log(h(\eta)+e)]^{\frac{N(1-\tilde{\rho})}{2}}
\int_{h(\eta)}^{\infty}\frac{d\sigma}{(\sigma+e)[\log(\sigma+e)]^{1+\frac{N(1-\tilde{\rho})}{2}}}
=\frac{C}{1-\tilde{\rho}}.
\end{multline*}
By Theorem~\ref{TH1}~(ii) we see that (\ref{S1E1}) has a nonnegative solution and that (\ref{TH3E1+}) holds.
\end{proof}

\begin{corollary}\label{S3C1}
Let $u_0\ge 0$.
Suppose that $f\in C[0,\infty)$, $f$ is nonnegative and $f$ is nondecreasing for $u>0$.
Suppose that $J\in C^2[0,\infty)$ satisfies (\ref{J}) and
\begin{equation}\label{S3C1E-1}
\textrm{the limit}\ \gamma:=\lim_{u\to\infty}\frac{J(u)J''(u)}{J'(u)^2}\ \textrm{exists.}
\end{equation}
If there exists a small $\e>0$ such that
\begin{equation}\label{S3C1E1}
\limsup_{u\to\infty}\frac{f(u)J'(u)}{J(u)^{1+2/N-\e}}<\infty,
\end{equation}
then (\ref{S1E1}) admits a local in time nonnegative solution for $u_0$ satisfying $J(u_0)\in L^1_{\rm ul}(\RN)$.
\end{corollary}
\begin{proof}
We define $q_J$ by
\[
q_J:=\lim_{u\to\infty}\frac{J'(u)^2}{J(u)J''(u)},
\]
which is a conjugate exponent of the growth rate of $J$.
In \cite{FI18,M18} it was shown that $1\le q_J\le\infty$ if the limit $q_J$ exists.
Therefore, $0\le \gamma\le 1$.

Let $\hf(u):=J(u)^{1+2/N-\e}/J'(u)$.
First, we consider the case $0<\gamma\le 1$.
Since $0<\gamma\le 1$, we can take $\theta\in (0,1]$ such that $1-\gamma<\theta<1-\gamma+2/N-\e$.
Then, for large $u>0$,
\begin{align}
\frac{d}{du}\left(\frac{f(u)}{J(u)^{\theta}}\right)
&=J(u)^{1+2/N-\theta-\e}\left(1+\frac{2}{N}-\theta-\e-\frac{J(u)J''(u)}{J'(u)^2}\right)>0,\label{S3C1E1+-}\\
\frac{d}{du}\left(\frac{J'(u)}{J(u)^{1-\theta}}\right)
&=\frac{J'(u)^2}{J(u)^{2-\theta}}\left(\frac{J(u)J''(u)}{J'(u)^2}-1+\theta\right)>0.\nonumber
\end{align}
Let $\xi$ be large, and let $\tf$ and $\tJ$ be defined by (\ref{TH1E0}).
Then, there exists $\xi>0$ such that
\begin{equation}\label{S3C1E1+}
\tf(u)=\frac{\hf(u)}{J(u)^{\theta}}\ \textrm{for}\ u\ge\xi\ \ 
\textrm{and}\ \ \tJ(u)=\frac{J'(u)}{J(u)^{1-\theta}}\ \textrm{for}\ u\ge\xi.
\end{equation}
Second, we consider the case $\gamma=0$.
Let $\theta =1$.
Then we see that (\ref{S3C1E1+-}) holds for large $\tau>0$.
Since $\theta=1$,
\[
\frac{d}{du}\left(\frac{J'(u)}{J(u)^{1-\theta}}\right)=J''(u)\ge 0
\ \ \textrm{for large}\ u>0.
\]
Then there exists $\xi>0$ such that (\ref{S3C1E1+}) holds.

We prove the corollary in the case $0\le \gamma\le1$.
It follows from (\ref{S3C1E1}) that
\begin{equation}\label{S3C1E2}
f(u)\le C_0\hf(u)\ \ \textrm{for large}\ u>0.
\end{equation}
Since
\[
\tJ(\eta)\int_{\eta}^{\infty}\frac{\tf(\tau)J'(\tau)d\tau}{J(\tau)^{1+2/N}}
=\int_{\eta}^{\infty}\frac{J(\tau)^{1+2/N-\theta-\e}J'(\tau)^2d\tau}{J'(\tau)J(\tau)^{2+2/N-\theta}}
=\int_{\eta}^{\infty}\frac{J'(\tau)d\tau}{J(\tau)^{1+\e}}=\frac{1}{\e}J(\eta)^{-\e}\to 0
\ \ \textrm{as}\ \ \eta\to\infty,
\]
by (\ref{S3C1E2}) we see that (\ref{TH1E1}) holds for $f$.
Hence, it follows from Theorem~\ref{TH1}~(i) that (\ref{S1E1}) with $f$ admits a nonnegative solution if $J(u_0)\in L^1_{\rm ul}(\RN)$.
The proof is complete.
\end{proof}

\begin{corollary}\label{S3C2}
Let $u_0\ge 0$.
Suppose that $f\in C[0,\infty)$, $f$ is nonnegative and $f$ is nondecreasing for $u>0$.
Suppose that $J\in C^2[0,\infty)$ satisfies (\ref{J}) and
\begin{equation}\label{S3C2E-1}
\lim_{u\to\infty}\frac{J(u)J''(u)}{J'(u)^2}=0.
\end{equation}
If there exists $\gamma>N/2$ such that
\begin{equation}\label{S3C2E0}
\limsup_{u\to\infty}\frac{f(u)J'(u)[\log(J(u)+e)]^{2\gamma/N}}{J(u)^{1+2/N}}<\infty,
\end{equation}
then (\ref{S1E1}) admits a local in time nonnegative solution for $u_0$ satisfying $J(u_0)\in L^1_{\rm ul}(\RN)$.
\end{corollary}

\begin{proof}
Let $\hf(u)=J(u)^{1+2/N}/J'(u)[\log(J(u)+e)]^{2\gamma/N}$ and let $\theta=1$.
By the same argument as in the proof of Corollary~\ref{S3C1} we see that $\tJ(u)=J'(u)/J(u)^{1-\theta}$ for $u\ge \xi$ if $\xi>0$ is large.
Since
\[
\frac{d}{d\tau}\left(\frac{J(u)^{2/N}}{J'(u)[\log(J(u)+e)]^{2\gamma/N}}\right)
=\frac{J(u)^{2/N-1}}{\log(J(u)+e)}\left(\frac{2}{N}-\frac{J(u)J''(u)}{J'(u)^2}-\frac{2\gamma}{N}\frac{J(u)}{(J(u)+e)\log(J(u)+e)}\right)>0
\]
for large $u>0$, we see that $\tf(u)=\hf(u)/J(u)^{\theta}$ for $u\ge\xi$ if $\xi>0$ is large.
Thus, there exists $\xi>0$ such that (\ref{S3C1E1+}) holds.

It follows from (\ref{S3C2E0}) that
\begin{equation}\label{S3C2E1}
f(u)\le C_0\hf(u)\ \ \textrm{for large}\ u>0.
\end{equation}
Since
\begin{multline*}
\tJ(\eta)\int_{\eta}^{\infty}\frac{\tf(\tau)J'(\tau)d\tau}{J(\tau)^{1+2/N}}
\le
\int_{\eta}^{\infty}\frac{J(\tau)^{1+2/N-\theta}J'(u)^2d\tau}{J(u)^{2+2/N-\theta}J'(u)[\log(J(u)+e)]^{2\gamma/N}}\\
\le\frac{C}{2\gamma/N-1}[\log(J(\eta)+e)]^{1-2\gamma/N}
\to 0\ \ \textrm{as}\ \ \eta\to\infty.
\end{multline*}
by (\ref{S3C2E1}) we see that (\ref{TH1E1}) holds for $f$.
Hence, it follows from Theorem~\ref{TH1}~(i) that (\ref{S1E1}) with $f$ admits a nonnegative solution if $J(u_0)\in L^1_{\rm ul}(\RN)$.
The proof is complete.
\end{proof}

\section{Nonexistence}
In this section let $N\ge1$ and $0\le\alpha<N/2$.
We begin to consider the case where $f(u)=f_{\beta}(u)$, $\beta>0$. 
We recall that
\[
F_\beta(u):=\int_u^{\infty}\frac{d\tau}{f_\beta(\tau)}\ \ \textrm{and}\ \
h_\beta(u):=F_\beta(u)^{-N/2}.
\]
Let $\e\in (0,N/2-\alpha)$.
There exists $C_0>0$ such that $f_\beta(u)$ is convex on $[C_0,\infty)$.
Then there exists $m\in (0,1/e)$ such that $|x|^{-N}(\log(1/|x|))^{-N/2-1+\e}\ge h_{\beta}(C_0)$ for $|x|\le m$.
We define
\begin{equation}\label{S4E1}
u_0 (x):=
\begin{cases}
h^{-1}_\beta\left(|x|^{-N}\left(\log\frac{1}{|x|}\right)^{-N/2-1+\e}\right) & \text{if $|x|\le m$,}\\
h^{-1}_\beta\left(m^{-N}\left(\log\frac{1}{m}\right)^{-N/2-1+\e}\right) & \text{if $|x|>m$.}
\end{cases}
\end{equation}
\begin{lemma}\label{S4L1}
Let $J_{\alpha}(u)$ and $g(u)$ be defined by (\ref{THBE0}).
Let $f(u)=f_{\beta}(u)$, $\beta>0$ and let $u_0(x)$ be defined by (\ref{S4E1}).
Then, the following hold:
\begin{enumerate}
\item $J_\alpha(u_0)=g(h_\beta(u_0))\in L^1_{\rm ul}(\RN)$.
\item $u_0\in\calL^1_{\rm ul}(\RN)$.
\end{enumerate}
\end{lemma}
\begin{proof}
(i) Let $\rho\in (0,m]$ be fixed.
It suffices to prove $\int_{B(0,\rho)} J_\alpha(u_0(x)) dx<\infty$.
If $|x|\le m$, then
\begin{eqnarray}
\label{eq14+}
\log\left(|x|^{-N}\left(\log\frac{1}{|x|}\right)^{-N/2-1+\e}+e\right)
\le\log\left(|x|^{-N}+e\right)\le2N\log\frac{1}{|x|},
\end{eqnarray}
which yields
\[
J_\alpha(u_0(x))=g\left(|x|^{-N}\left(\log\frac{1}{|x|}\right)^{-N/2-1+\e}\right)
\le|x|^{-N}\left(\log\frac{1}{|x|}\right)^{-N/2-1+\e}\left(2N\log\frac{1}{|x|}\right)^{\alpha}.
\]
We deduce that
\[
\displaystyle\int_{B(0,\rho)}J_\alpha(u_0(x)) dx\lesssim
\displaystyle\int_0^\rho \frac{1}{r}\left(\log\frac{1}{r}\right)^{-N/2-1+\e+\alpha}dr
=\frac{1}{N/2-\e-\alpha}\cdot\left(\log\frac{1}{\rho}\right)^{-N/2+\e+\alpha}<\infty.
\]
Thus, $J_\alpha(u_0)\in L^1_{\rm ul}(\RN)$.\\
(ii) Let $\{\phi_n(x)\}_{n=1}^{\infty}$ be defined by
\[
\phi_n (x):=
\begin{cases}
u_0(2^{-n}m) & \text{if $|x|\le 2^{-n}m$,}\\
u_0(x) & \text{if $|x|>2^{-n}m$.}
\end{cases}
\]
We see that $\{\phi_n\}\subset {BUC}(\RN)$. 
Since $f'(u)F(u)\le q$ for large $u>0$, $h^{-1}_\beta(u)$ is concave for large $u>0$, which implies that $h^{-1}_\beta(u)\lesssim u$ for large $u>0$.
Therefore, if $n$ is large, then
\begin{multline*}
\displaystyle\int_{\RN} |u_0(x)-\phi_n (x)| dx
\le\displaystyle\int_{B(0,2^{-n}m)} h^{-1}_\beta\left(|x|^{-N}\left(\log\frac{1}{|x|}\right)^{-N/2-1+\e}\right)dx\\
\lesssim
\displaystyle\int_0^{2^{-n}m} \frac{1}{r}\left(\log\frac{1}{r}\right)^{-N/2-1+\e}dr
=\frac{1}{N/2-\e}\left(\log\frac{2^n}{m}\right)^{-N/2+\e}
\to0 \ \ \text{as $n\to\infty$.}
\end{multline*}
Thus, $u_0\in\mathcal{L}^1_{\rm ul}(\RN)$.
\end{proof}

\begin{theorem}\label{S4T1}
Let $N \ge 1$, $0\le\alpha<N/2$ and $\beta>0$.
Let $u_0$ be defined by (\ref{S4E1}).
Then (\ref{S1E1}) with $f(u)=f_{\beta}(u)$ admits no local in time nonnegative solution.
\end{theorem}
We postpone the proof of Theorem~\ref{S4T1}.

\begin{proof}[\bf Proof of Theorem \ref{THC}]
We construct an initial function $u_0$. Choose $\e$ such that 
$0<\e<N/2-\max\{\alpha,\delta'\}$, where $0<\delta'<N/2$ is chosen later. 
We also choose $C_0\ge C_1$ such that $f_\beta(u)$ is convex on $[C_0,\infty)$.
Then we define $v_0(x)$ by the right hand side of \eqref{S4E1}.
Put $\mathcal{J}:=F^{-1}\circ F_\beta$ and $u_0(x):=\mathcal{J}(v_0(x))$.
We see that $u_0(x)=h^{-1}\left(|x|^{-N}\left(\log\frac{1}{|x|}\right)^{-N/2-1+\e}\right)$
for $|x|\le{m}$.
Then we can obtain $J_{\alpha}(u_0)\in L^1_{\rm ul}(\RN)$ in the same way as Lemma~\ref{S4L1}~(i).

We show that $u_0\in L^1_{\rm ul}(\RN)$.
It suffices to prove $\int_{B(0,\rho)} u_0(x) dx<\infty$ for small $\rho>0$.
By the assumption (ii) we can apply Lemma~\ref{S2L3}.
Using Lemma~\ref{S2L3} and \eqref{eq14+} we see that there exists $\delta'\in(0,N/2)$ independent of $\e$ such that
\[
h^{-1}\left(|x|^{-N}\left(\log\frac{1}{|x|}\right)^{-N/2-1+\e}\right)
\lesssim |x|^{-N}\left(\log\frac{1}{|x|}\right)^{-N/2-1+\e}\left(2N\log\frac{1}{|x|}\right)^{\delta'}
\]
for small $|x|>0$. Then we deduce that
\[
\displaystyle\int_{B(0,\rho)} u_0(x) dx\lesssim
\displaystyle\int_0^\rho \frac{1}{r}\left(\log\frac{1}{r}\right)^{-N/2-1+\e+\delta'}dr
=\frac{1}{N/2-\e-\delta'}\cdot\left(\log\frac{1}{\rho}\right)^{-N/2+\e+\delta'}
\]
for small $\rho>0$. Thus, $u_0\in L^1_{\rm ul}(\RN)$.

The proof is by contradiction. Suppose that there exists $T>0$ such that \eqref{S1E1} has a nonnegative solution. Since $u(t)\ge 
S(t)u_0\ge\calJ(C_0)$ in $\RN\times(0,T)$, we can define $v(x,t):=\calJ^{-1}(u(x,t))$.
Since $\calJ''(v)\ge 0$ for $v\ge C_0$ by the assumption (i), we have $v(t)\ge C_0$ and
\[
\partial_t v=\frac{\calJ''(v)}{\calJ'(v)}|\nabla v|^2+\Delta v+\frac{f(\calJ(v))}{\calJ'(v)}\ge\Delta v+f_\beta(v) \ \ \text{in $\RN\times(0,T)$.}
\]
Here we use $f(\calJ(v))=f_{\beta}(v)\calJ'(v)$.
Then it follows (see \cite[p.77]{QS07} for details) that
\begin{equation}\label{THCPE01}
v(t)\ge S(t-\tau)v(\tau)+\int_{\tau}^t S(t-s)f_\beta(v(s))ds
\ \ \text{in $\RN\times(\tau,T)$.}
\end{equation}
We also obtain
\[
\mathcal{J}(v(\tau))=u(\tau)\ge S(\tau)u_0=S(\tau)\mathcal{J}(v_0)
\ge\mathcal{J}(S(\tau)v_0) \ \ \text{in $\RN\times(0,T)$.}
\]
Since $\mathcal{J}$ is increasing on $[C_0,\infty)$, we see that $v(\tau)\ge S(\tau)v_0$ in $\RN\times(0,T)$.
Letting $\tau\to 0$ in (\ref{THCPE01}), we have
\[
v(t)\ge S(t)v_0+\int_0^t S(t-s)f_\beta(v(s))ds
\ \ \text{in $\RN\times(0,T)$,}
\]
and hence Proposition~\ref{S2P1} says that (\ref{S1E1}) has a nonnegative solution.
It contradicts the nonexistence result in Theorem~\ref{S4T1}.
We complete the proof.
\end{proof}

\begin{proof}[\bf Proof of Theorem \ref{S4T1}]
The proof is by contradiction.
Suppose that there exists $T>0$ such that (\ref{S1E1}) with $f(u)=f_{\beta}(u)$ possesses a local in time nonnegative solution on $(0,T)$.
Let $0<\xi<t<T$. It follows from the Fubini theorem and \eqref{S1D1E1} that
\begin{equation}\label{eq15}
u(t)=S(t-\xi)u(\xi)+\displaystyle\int_\xi^t S(t-s)f_\beta(u(s))ds.
\end{equation}
By \eqref{S1D1E1} and $u_0\ge C_0$ we have $u(t)\ge S(t)u_0\ge C_0$ in $\RN\times(0,T)$.
Let $G(x.t):=K(x,0,t)$.
Then, by \eqref{eq15} we can obtain in a similar way to \cite[Eq. (3.22)]{HI18} that
if $\rho>0$ is sufficiently small, then
\begin{eqnarray}
\label{eq16}
w(t)\ge c_* M_\tau 3^{-N/2} G(0,1) t^{-N/2}
+2^{-N/2} t^{-N/2}\displaystyle\int_{\rho^2}^t s^{N/2}f_\beta(w(s)) ds
\end{eqnarray}
holds for almost all $0<\tau<\rho^2$ and $\rho^2<t<(T-4\rho^2)/3$, where $c_*>0$ is a constant depending only on $N$,
\[
w(t):=\int_\RN u(x,t+4\rho^2)G(x,t) dx
\quad\text{and}\quad M_\tau:=\displaystyle\int_{B(0,\rho)} u(y,\tau) dy.
\]

We show that if $0<\rho<1$ is sufficiently small, then there exist $C_1>0$ and $\delta>0$
such that
\begin{eqnarray}
\label{eq17}
w(s)\ge C_1 s^{-N/2}\left(\log\frac{2}{3\delta^2\rho^2}\right)^{-N(\beta+1)/2-1+\e}
\ \ \text{for $\rho^2<s<\rho$.}
\end{eqnarray}
Let $0<\rho<1$ and $\rho^2<s<\rho$.
We see that
\begin{align}\label{eq18}
\begin{split}
w(s)&\ge\displaystyle\int_{\RN} S{(s+4\rho^2)}[u_0](x) G(x,s) dx\\
&=\displaystyle\int_{\RN}\displaystyle\int_{\RN} G(x-y,s+4\rho^2)u_0(y)dy\ G(x,s)dx\\
&=\displaystyle\int_{\RN}\displaystyle\int_{\RN} G(x-y,s+4\rho^2)G(x,s)dx\ u_0(y)dy\\
&=\displaystyle\int_{\RN}\displaystyle\int_{\RN} G(y-x,s+4\rho^2)G(x,s)dx\ u_0(y)dy\\
&=\displaystyle\int_{\RN} G(y,2s+4\rho^2)u_0(y)dy.
\end{split}
\end{align}
Here we use $G(y,t)=\int_{\RN} G(y-x,t-s)G(x,s)dx$ for $y\in\RN$ and $0<s<t$.
By Lemma~\ref{S2L2} there exists a small $m'>0$ such that for $|y|\le m'$,
\begin{equation}\label{eq18+}
h^{-1}_\beta\left(|y|^{-N}\left(\log\frac{1}{|y|}\right)^{-N/2-1+\e}\right)
\ge\tilde{h}_\beta\left(|y|^{-N}\left(\log\frac{1}{|y|}\right)^{-N/2-1+\e}\right)
\gtrsim|y|^{-N}\left(\log\frac{1}{|y|}\right)^{-N(\beta+1)/2-1+\e}.
\end{equation}
Put $s_*:=2s+4\rho^2$ and choose $\delta>0$ such that
$\sqrt{6}\delta\le\min\{m,m'\}$. Then we obtain
\begin{align}
\label{eq19}
\begin{split}
&\displaystyle\int_{\RN} G(y,2s+4\rho^2)u_0(y)dy\\
&\hspace{30pt}\gtrsim(4\pi s_*)^{-N/2}\displaystyle\int_{\{|y|\le\sqrt{s_*}\delta\}} 
e^{-\frac{|y|^2}{4s_*}}\cdot|y|^{-N}\left(\log\frac{1}{|y|}\right)^{-N(\beta+1)/2-1+\e}dy\\
&\hspace{30pt}\ge(4\pi)^{-N/2}\displaystyle\int_{\{|z|\le\delta\}} 
e^{-\frac{|z|^2}{4}}\cdot{s_*}^{-N/2}|z|^{-N} \left(\log\frac{1}{\sqrt{s_*}|z|}\right)^{-N(\beta+1)/2-1+\e}dz\\
&\hspace{30pt}\ge(4\pi)^{-N/2}\displaystyle\int_{\{\delta/2\le|z|\le\delta\}} 
e^{-\frac{|z|^2}{4}}\cdot{s_*}^{-N/2}|z|^{-N} \left(\log\frac{2}{\sqrt{s_*}\delta}\right)^{-N(\beta+1)/2-1+\e}dz\\
&\hspace{30pt}\gtrsim {s_*}^{-N/2}\left(\log\frac{4}{s_*\delta^2}\right)^{-N(\beta+1)/2-1+\e}.
\end{split}
\end{align}
Here we put $y=\sqrt{s_*}z$. By \eqref{eq18} and \eqref{eq19} we have
\begin{eqnarray}
\label{eq20}
s^{N/2}w(s)\gtrsim\left(\frac{s}{s_*}\right)^{N/2}\left(\log\frac{4}{s_*\delta^2}\right)^{-N(\beta+1)/2-1+\e}.
\end{eqnarray}
Since the right hand side of \eqref{eq20} is nondecreasing with respect to $s$, we have
\[
s^{N/2}w(s)\gtrsim\left(\frac{1}{6}\right)^{N/2}\left(\log\frac{2}{3\delta^2\rho^2}\right)^{-N(\beta+1)/2-1+\e}.
\]
Thus there exists $C_1>0$ such that \eqref{eq17} holds. We observe from $s<\rho$ that
\[
s^{-N/2}\left(\log\frac{2}{3\delta^2\rho^2}\right)^{-N(\beta+1)/2-1+\e}>
\rho^{-N/2}\left(\log\frac{2}{3\delta^2\rho^2}\right)^{-N(\beta+1)/2-1+\e}
\to\infty \ \ \text{as $\rho\to0$.}
\]
Hence, we choose $\rho$ sufficiently small and the right hand side of \eqref{eq17} is greater than $1$. Then we have
\begin{equation}\label{eq21}
\log(w(s)+e)\ge\log\left\{C_1s^{-N/2}\left(\log\frac{2}{3\delta^2\rho^2}\right)^{-N(\beta+1)/2-1+\e}
\right\}=\frac{N}{2}\left(\log\frac{1}{s}+C_2(\rho)\right)
\end{equation}
for $\rho^2<s<\rho$, where $C_2(\rho):=\frac{2}{N}\log C_1+\frac{2}{N}(-N(\beta+1)/2-1+\e)\log(\log\frac{2}{3\delta^2\rho^2})$.
By \eqref{eq16} and \eqref{eq21} we have
\begin{align}\label{eq22}
\begin{split}
t^{N/2}w(t)&\ge c_* M_\tau 3^{-N/2} G(0,1)+2^{-N/2}\left(\frac{N}{2}\right)^{\beta}
\displaystyle\int_{\rho^2}^t s^{N/2}\left(\log\frac{1}{s}+C_2(\rho)\right)^{\beta} w(s)^{1+2/N} ds\\
&=:H(t)
\end{split}
\end{align}
for almost all $0<\tau<\rho^2$ and $\rho^2<t<\rho$. Note that 
$\rho<(T-4\rho^2)/3$ holds since $\rho$ is sufficiently small.
Put $C_3:=2^{-N/2}(N/2)^{\beta}$. By \eqref{eq22} we have
\begin{multline}\label{eq22+}
\frac{dH}{dt}(t)
=C_3 t^{N/2}\left(\log\frac{1}{t}+C_2(\rho)\right)^{\beta} w(t)^{1+2/N}\\
\ge C_3 t^{N/2}\left(\log\frac{1}{t}+C_2(\rho)\right)^{\beta} (t^{-N/2}H(t))^{1+2/N}
=\frac{C_3}{t}\left(\log\frac{1}{t}+C_2(\rho)\right)^{\beta} H(t)^{1+2/N},
\end{multline}
which yields
\[
-\frac{N}{2}\cdot\frac{d}{dt}\left(H(t)^{-2/N}\right)
\ge-\frac{C_3}{\beta+1}\cdot\frac{d}{dt}\left\{\left(\log\frac{1}{t}+C_2(\rho)\right)^{\beta+1}\right\}.
\]
This implies that
\[
\frac{N}{2}H(\rho^2)^{-2/N}-\frac{N}{2}H(\rho)^{-2/N}
\ge\frac{C_3}{\beta+1}\left\{\left(2\log\frac{1}{\rho}+C_2(\rho)\right)^{\beta+1}-\left(\log\frac{1}{\rho}+C_2(\rho)\right)^{\beta+1}\right\}.
\]
Since $H(\rho^2)=c_* M_\tau 3^{-N/2} G(0,1)$ and $H(\rho)>0$, we have
\begin{eqnarray}
\label{eq23}
M_\tau\lesssim\left\{\left(2\log\frac{1}{\rho}+C_2(\rho)\right)^{\beta+1}-\left(\log\frac{1}{\rho}+C_2(\rho)\right)^{\beta+1}\right\}^{-N/2}
\end{eqnarray}
for almost all $0<\tau<\rho^2$.
We see that 
\begin{eqnarray}
\label{eq34}
M_\tau=\displaystyle\int_{B(0,\rho)} u(y,\tau) dy\ge\displaystyle\int_{B(0,\rho)} S(\tau)[u_0](y) dy.
\end{eqnarray}
By Lemma~\ref{S4L1}~(ii) and
Proposition~\ref{S2P3} we see that $\|S(\tau)u_0-u_0\|_{L^1_{\rm ul}(\RN)}\to0$ as $\tau\to0$.
Then there exists a subsequence
$\{(S(\tilde{\tau})u_0)|_{B(0,\rho)}\}_{\tilde{\tau}}$ such that
$(S(\tilde{\tau})u_0)|_{B(0,\rho)}\to u_0|_{B(0,\rho)}$ as $\tilde{\tau}\to0$
a.e. in $B(0,\rho)$.
Thus it follows from Fatou's lemma, \eqref{eq23} and \eqref{eq34} that
\begin{equation}\label{eq35}
\left\{\left(2\log\frac{1}{\rho}+C_2(\rho)\right)^{\beta+1}-\left(\log\frac{1}{\rho}+C_2(\rho)\right)^{\beta+1}\right\}^{-N/2}
\gtrsim\liminf_{\tilde{\tau}\to 0}\int_{B(0,\rho)}S(\tilde{\tau})[u_0](y)dy
\ge\int_{B(0,\rho)} u_0(y) dy.
\end{equation}
On the other hand, by \eqref{eq18+} we have
\begin{multline}\label{eq25}
\int_{B(0,\rho)} u_0(y)dy\ge
\int_{B(0,\rho)} |y|^{-N}\left(\log\frac{1}{|y|}\right)^{-N(\beta+1)/2-1+\e}dy\\
=\int_0^{\rho} \frac{1}{r}\left(\log\frac{1}{r}\right)^{-N(\beta+1)/2-1+\e}dr
=\frac{1}{N(\beta+1)/2-\e}\cdot\left(\log\frac{1}{\rho}\right)^{-N(\beta+1)/2+\e}
\end{multline}
for $0<\rho\le2\min\{m,m'\}$. By \eqref{eq35} and \eqref{eq25} we have
\begin{eqnarray}
\label{eq26}
\left\{\left(\frac{2\log\frac{1}{\rho}+C_2(\rho)}{\log\frac{1}{\rho}}
\right)^{\beta+1}-\left(\frac{
\log\frac{1}{\rho}+C_2(\rho)}{\log\frac{1}{\rho}}
\right)^{\beta+1}\right\}^{-N/2}\gtrsim\left(\log\frac{1}{\rho}\right)^{\e}\to\infty
\ \ \text{as $\rho\to0$.}
\end{eqnarray}
Here we use $C_2(\rho)/\log(1/\rho)\to0$ as $\rho\to0$.
Then the left hand side of \eqref{eq26} converges to $(2^{\beta+1}-1)^{-N/2}$ as $\rho\to0$.
This is a contradiction.
The proof is complete.
\end{proof}

Using Proposition~\ref{S1P2}~(ii), we can obtain a nonexistence result corresponding to Theorem~\ref{TH1}.

\begin{corollary}\label{S4C1}
Suppose that $f\in C[0,\infty)$, $f$ is nonnegative and $f$ is nondecreasing for $u>0$.
Suppose that $J\in C^3[0,\infty)$ satisfies (\ref{J}) and
\begin{equation}\label{S4C1E0-}
\textrm{the limit }
\delta:=\lim_{u\to\infty}\frac{J(u)J'''(u)}{J'(u)J''(u)}\ \textrm{exists}.
\end{equation}
If there exists $\e>0$ such that
\begin{equation}\label{S4C1E0}
\liminf_{u\to\infty}\frac{f(u)J'(u)}{J(u)^{1+2/N+\e}}>0,
\end{equation}
then there exists $u_0\ge 0$ such that $J(u_0)\in L^1_{\rm ul}(\RN)$ and (\ref{S1E1}) admits no nonnegative solution.
\end{corollary}

\begin{proof}
Because of (\ref{S4C1E0}), there exist $C_0>0$ and $C_1>0$ that
\begin{equation}\label{S4C1E0+}
f(u)>C_0\frac{J(u)^{1+2/N+\e}}{J'(u)}\ \textrm{for}\ u\ge C_1.
\end{equation}
Here, $C_0>0$ can be arbitrary large, since $C_1>0$ can be arbitrary large and $\e>0$ can be arbitrary small.
Let $\hf(u):=\rho J(u)^{1+1/\rho}/J'(u)$, $0<\rho<N/2$.
Here, $\rho$ is determined later.
Then $\hF(u):=\int_u^{\infty}d\tau/\hf(\tau)=J(u)^{-1/\rho}$.
First, we consider the Cauchy problem
\begin{equation}\label{S4C1E1}
\begin{cases}
\partial_t u=\Delta u+\hf (u) & \textrm{in}\ \RN\times (0,T),\\
u(x,0)=u_0(x) & \textrm{in}\ \RN.
\end{cases}
\end{equation}
By direct calculation we have
\begin{equation}\label{S4C1E1+-}
\hf'(u)\hF (u)=\rho+1-\rho\frac{J(u)J''(u)}{J'(u)^2}\to \rho+1-{\rho\gamma}
\ \ \textrm{as}\ \ u\to\infty.
\end{equation}
By (\ref{S4C1E1+-}) we see that $\hf'(u)>0$ for large $u>0$.
We can take $C_1>0$ such that $\hf'(u)>0$ for $u>C_1$.
Then, we can modify $\hf(u)$, $0\le u\le C_1$, such that $\hf$ satisfies (\ref{f}).
Hereafter, we do not use $\hf(u)$ in $0\le u\le C_1$. 

By L'Hospital's rule we see the following limit $\gamma$ exists:
\[
\gamma=\lim_{u\to\infty}\frac{J(u)J''(u)}{J'(u)^2}=\lim_{u\to\infty}\left(\frac{1}{2}+\frac{J(u)J'''(u)}{2J'(u)J''(u)}\right)=\frac{1}{2}+\frac{\delta}{2}.
\]
Moreover,
\[
\eta=\lim_{u\to\infty}\frac{J(u)^2J'''(u)}{J'(u)^3}=\lim_{u\to\infty}\frac{J(u)J''(u)}{J'(u)^2}\frac{J(u)J'''(u)}{J'(u)J''(u)}=\delta\gamma.
\]
Therefore, $\eta=\gamma-2\gamma^2$.
By direct calculation we have
\begin{multline*}
\lim_{u\to\infty}\frac{\hf''(u)}{\rho J(u)^{1+1/\rho}J'(u)}\\
=\lim_{u\to\infty}\left(
\frac{1}{\rho}\left(1+\frac{1}{\rho}\right)-\left(1+\frac{1}{\rho}\right)\frac{J(u)J''(u)}{J'(u)^2}
+2\left(\frac{J(u)J''(u)}{J'(u)^2}\right)^2-\frac{J(u)^2J'''(u)}{J'(u)^3}
\right)
=\frac{1}{\rho}(1-\gamma)+\frac{1}{\rho^2}> 0,
\end{multline*}
and hence $\hf''(u)\ge 0$ for large $u>0$.
Here, we see that $0\le\gamma\le 1$ as mentioned in the proof of Corollary~\ref{S3C1}.
We can retake $C_1>0$ such that $\hf''(u)>0$ for $u>C_1$.
Then, we again modify $\hf(u)$, $0\le u\le C_1$, such that $\hf''(u)$ for $u\ge 0$.
Hereafter, we do not use $\hf(u)$ in $0\le u\le C_1$. 

Since $\hf\in C^2$, by (\ref{S4C1E1+-}) we see that $\hf\in C^2[0,\infty)\cap X_q$ with $q=1+\rho(1-\gamma)$.
We see that $1<q<1+N/2$.
We have checked all the assumptions of Proposition~\ref{S1P2}~(ii).
It follows from Proposition~\ref{S1P2}~(ii) that, for each $r\in [q-1,N/2)$, there exists a nonnegative function $u_0\in {L}^1_{\rm ul}(\RN)$ such that $\hF(u_0)^{-r}\in L^1_{\rm ul}(\RN)$ and (\ref{S4C1E1}) admits no nonnegative solution.
Without loss of generality, we can assume that $u_0\ge C_1$.
Since $q-1\le \rho<N/2$, we can take $r=\rho$.
We have
\begin{equation}\label{S4C1E1+}
\hF(u_0)^{-r}=J(u_0)\in L^1_{\rm ul}(\RN).
\end{equation}

Second, we consider (\ref{S1E1}).
We take $\rho=N/(N\e+2)$.
Then, $q-1=\rho(1-\gamma)\le\rho <N/2$, and hence all the conditions before are satisfied.
Because of (\ref{S4C1E0+}), we have
\[
f(u)>C_0\frac{J(u)^{1+2/N+\e}}{J'(u)}>\rho \frac{J(u)^{1+1/\rho}}{J'(u)}=\hf(u)
\ \ \textrm{for}\ \ u\ge C_1.
\]
Suppose that (\ref{S1E1}) with the initial function $u_0$ has a solution $u(t)$.
Then,
\[
u(t)=S(t)u_0+\int_0^tS(t-s)f(u(s))ds\ge S(t)u_0+\int_0^tS(t-s)\hf(u(s))ds,
\]
and hence $u(t)$ is a supersolution for (\ref{S4C1E1}).
By Proposition~\ref{S2P1} we see that (\ref{S4C1E1}) has a nonnegative solution.
However, it contradicts the nonexistence of a nonnegative solution of (\ref{S4C1E1}) with $u_0$.
Thus, (\ref{S1E1}) with $u_0$ admits no nonnegative solution.
By (\ref{S4C1E1+}) we obtain the conclusion of the corollary.
\end{proof}

\begin{remark}\label{S4R1}
\mbox{}
\begin{enumerate}
\item When $f(u)=J(u)^{1+2/N}/J'(u)$, it follows from Theorem~\ref{THA} that (\ref{S1E1}) admits a local in time nonnegative solution for every $u_0\ge 0$ satisfying $J(u_0)\in\mathcal{L}^1_{\rm ul}(\RN)$.
Therefore,  we cannot take $\e=0$ in Corollary~\ref{S4C1}.
\item Remark~\ref{S4R1}~(i) indicates that a threshold growth is $f(u)=J(u)^{1+2/N}/J'(u)$ when the integrability condition is $J(u_0)\in\mathcal{L}^1_{\rm ul}(\RN)$.
On the other hand, Theorem~\ref{THA} and Proposition~\ref{S1P2}~(ii) indicate that a threshold integrability condition is $F(u_0)^{-N/2}\in\mathcal{L}^1_{\rm ul}(\RN)$ when the growth is $f(u)$.
\item Since $q:=1+\rho(1-\gamma)<1+N/2$, the extremal case is $\gamma=0$ and $\rho=N/2$.
In this case the $q$ exponent of $\hf(u)$, which is given by (\ref{S4C1E1+-}), is $1+N/2$.
Since $(q,r)=(1+N/2,N/2)$, (\ref{S4C1E1}) is a doubly critical case.
\end{enumerate}
\end{remark}

Using Theorem~\ref{THC}, we can obtain a nonexistence result corresponding to Theorem~\ref{TH1}.
\begin{corollary}\label{S4C2}
Let $g_{\gamma}(u)=u[\log(u+e)]^{\gamma}$.
Suppose that $f\in C[0,\infty)$, $f$ is nonnegative and $f$ is nondecreasing for $u>0$.
Suppose that $J\in C^2[0,\infty)$ satisfies (\ref{J}) and
\begin{equation}\label{S4C2E-1}
\lim_{u\to\infty}\frac{J(u)J''(u)}{J'(u)^2}=0.
\end{equation}
If there exists $\gamma\in (0,N/2)$ such that
\begin{equation}\label{S4C2E0}
\liminf_{u\to\infty}\frac{f(u)J'(u)\left[\log(J(u)+e)\right]^{2\gamma/N}}{J(u)^{1+2/N}}>0,
\end{equation}
and
\begin{equation}\label{S4C2E0+}
\frac{d^2}{du^2}\left(g_{\gamma}^{-1}(J(u))\right)\le 0\ \ \textrm{for large}\ \ u>0.
\end{equation}
then there exists a nonnegative function $u_0$ such that $J(u_0)\in L^1_{\rm ul}(\RN)$  and (\ref{S1E1}) admits no nonnegative solution.

In particular, if $J(u)=g_{\alpha}(u)$ for some $\alpha\in (0,N/2)$ and $\liminf_{u\to\infty}f(u)/u^{1+2/N}>0$, then there exists a nonnegative function $u_0$ such that $J(u_0)\in L^1_{\rm ul}(\RN)$ and (\ref{S1E1}) admits no nonnegative solution. 
\end{corollary}

\begin{proof}
Because of (\ref{S4C2E-1}), there exist $C_0>0$ and $C_1>1$ such that
\[
f(u)>C_0\frac{J(u)^{1+2/N}}{J'(u)[\log(J(u)+e)]^{2\gamma/N}}
\ \ \textrm{for}\ u\ge C_1.
\]
Note that, for each large $C_0>0$, we can retake $\gamma\in (0,N/2)$ and $C_1>0$ such that the above inequality holds.
Let $g_{\gamma}^{-1}(u)$ denote the inverse function of $g_{\gamma}$.
We define
\[
\hf(u):=\frac{N}{2}\frac{g_{\gamma}^{-1}(J(u))^{1+2/N}g_{\gamma}'(g_{\gamma}^{-1}(J(u)))}{J'(u)}.
\]
Then, $\hF(u):=\int_u^{\infty}d\tau/\hf(\tau)=g_{\gamma}^{-1}(J(u))^{-2/N}$.
First, we consider the Cauchy problem
\begin{equation}\label{S4C2E1}
\begin{cases}
\partial_tu=\Delta u+\hf(u) & \textrm{in}\ \RN\times (0,T),\\
u(x,0)=u_0(x) & \textrm{in}\ \RN.
\end{cases}
\end{equation}
By direct calculation we have
\begin{equation}\label{S4C2E2}
\hf'(u)\hF(u)
=1+\frac{N}{2}+\frac{N}{2}\frac{vg_{\gamma}'(v)}{g_{\gamma}(v)}
\left(
\frac{g_{\gamma}(v)g_{\gamma}''(v)}{g_{\gamma}'(v)^2}-
\frac{J(u)J''(u)}{J'(u)^2}
\right),
\end{equation}
where $v:=g_{\gamma}^{-1}(J(u))$.
As $v\to\infty$,
\begin{equation}\label{S4C2E2+}
\frac{vg_{\gamma}'(v)}{g_{\gamma}(v)}
=\frac{v[\log(v+e)]^{\gamma}(1+o(1))}{v[\log(v+e)]^{\gamma}}\to 1
\ \ \textrm{and}\ \ 
\frac{g_{\gamma}(v)g_{\gamma}''(v)}{g_{\gamma}'(v)^2}
=\frac{v[\log(v+e)]^{\gamma}\frac{\gamma}{v+e}[\log(v+e)]^{\gamma-1}(1+o(1))}{[\log(v+e)]^{2\gamma}(1+o(1)))}\to 0.
\end{equation}
By (\ref{S4C2E2+}), (\ref{S4C2E2}) and (\ref{S4C2E-1}) we have $\lim_{u\to\infty}\hf'(u)\hF(u)=1+N/2$.
Because of (\ref{S4C2E0+}), we have
\begin{equation}\label{S4C2E2++}
\frac{g_{\gamma}(v)g_{\gamma}''(v)}{g_{\gamma}'(v)^2}-\frac{J(u)J''(u)}{J'(u)^2}=
-\frac{g_{\gamma}(v)}{g_{\gamma}'(v)}\frac{\frac{d^2v}{du^2}}{\left(\frac{dv}{du}\right)^2}\ge 0.
\end{equation}
By (\ref{S4C2E2++}) and (\ref{S4C2E2}) we have
\begin{equation}\label{S4C2E2+++}
\hf'(u)\hF(u)=1+\frac{N}{2}-\frac{N}{2}\frac{v\frac{d^2v}{du^2}}{\left(\frac{dv}{du}\right)^2}\ge 1+\frac{N}{2}.
\end{equation}
By (\ref{S4C2E2+++}) we see that $\hf'(u)>0$ for large $u>0$.
Hence, we can take $C_1>0$ such that $\hf'(u)>0$ for $u>C_1$.
Moreover, we can modify $\hf(u)$, $0\le u\le C_1$, such that $\hf$ satisfies (\ref{f}).
Hereafter, we do not use $\hf(u)$ in $0\le u\le C_1$. 

Now we prove Remark~\ref{S1R3}~(i).
Assume that there exists $c>0$ such that $f'(F^{-1}(v))F(F^{-1}(v))\ge 1+N/2$ for $0<v\le c$.
Let $\beta>0$.
Since $f'_{\beta}(u)/f_{\beta}(u)^{1/(1+N/2)}$ is nondecreasing for $u>0$, we obtain in the same way as (\ref{S2E1E1}) that $f'_{\beta}(u)F_{\beta}(u)\le 1+N/2$ for $u>0$.
This implies that $f'_{\beta}(F_{\beta}^{-1}(v))F_{\beta}(F_{\beta}^{-1}(v))\le 1+N/2$ for $0<v<F_{\beta}(0)=\infty$.
Thus we obtain $f'(F^{-1}(v))\ge (1+N/2)/v\ge f'_{\beta}(F_{\beta}^{-1}(v))$ for $0<v\le c$.

By Remark~\ref{S1R3}~(i) and (\ref{S4C2E2+++}) we see that the assumption Theorem~\ref{THC}~(i) is satisfied.
Because of (\ref{J}), we see that $J(u)\ge C_2u$ for large $u>0$.
Since $0<\gamma<N/2$, there exist $\delta\in (0,1)$ and $C_3>0$ such that
\[
J(u)\ge C_2u\ge C_3u[\log(u+e)]^{\gamma-N\delta/2}
\ \ \textrm{for large}\ u>0.
\]
Then,
\[
\hF(u)=g_{\gamma}^{-1}(J(u))^{-\frac{2}{N}}
\le g_{\gamma}^{-1}(C_3u[\log(u+e)]^{\gamma-\frac{N\delta}{2}})^{-\frac{2}{N}}
\le C_4u^{-\frac{2}{N}}[\log(u+e)]^{\delta},
\]
and hence the assumption Theorem~\ref{THC}~(ii) is satisfied.
Using Theorem~\ref{THC}, we see that for each $\alpha\in [0,N/2)$, there exists $u_0\ge 0$ such that $J_{\alpha}(u_0)=g_{\alpha}(\hF(u_0)^{-N/2})\in L^1_{\rm ul}(\RN)$ and (\ref{S1E1}) admits no nonnegative solution.
We take $\alpha=\gamma$.
Then,
\begin{equation}\label{S4C2E4}
J(u)=J_{\gamma}(u)\in L^1_{\rm ul}(\RN).
\end{equation}

We consider (\ref{S4C2E1}).
Since
\begin{multline*}
\frac{N}{2}g_{\gamma}^{-1}(J(u))^{1+2/N}g_{\gamma}'(g_{\gamma}^{-1}(J(u)))
\le C_5(J(u)[\log(J(u)+e)]^{-\gamma})^{1+2/N}[\log(J(u)+e)]^{\gamma}\\
=C_5J(u)^{1+2/N}[\log(J(u)+e)]^{-2\gamma/N}\ \ \textrm{for large}\ u>0,
\end{multline*}
there exists $C_6>C_1$ such that $C_0>C_5$ and
\begin{multline*}
f(u)>C_0\frac{J(u)^{1+2/N}}{J'(u)[\log(J(u)+e)]^{2\gamma/N}}
\ge
C_5\frac{J(u)^{1+2/N}}{J'(u)[\log(J(u)+e)]^{2\gamma/N}}\\
\ge
\frac{N}{2}\frac{g_{\gamma}^{-1}(J(u))^{1+2/N}g_{\gamma}'(g_{\gamma}^{-1}(J(u)))}{J'(u)}
=\hf(u)
\ \ \textrm{for}\ u\ge C_6.
\end{multline*}
We can assume that $u_0\ge C_6$.
Suppose that (\ref{S1E1}) with the initial function $u_0$ has a solution $u(t)$. Then
\[
u(t)=S(t)u_0+\int_0^tS(t-s)f(u(s))ds
\ge S(t)u_0+\int_0^tS(t-s)\hf(u(s))ds,
\]
and hence $u(t)$ is a supersolution for (\ref{S4C2E1}).
By Proposition~\ref{S2P1} we see that (\ref{S4C2E1}) has a nonnegative solution.
However, it contradicts the nonexistence of a nonnegative solution of (\ref{S4C2E1}) with $u_0$.
Thus, (\ref{S1E1}) with $u_0$ admits no nonnegative solution.
By (\ref{S4C2E4}) we obtain the first statement of the corollary.

Next, we prove the second statement.
We take $\gamma=\alpha$.
Then, (\ref{S4C2E0+}) holds, and $\liminf_{u\to\infty}f(u)/u^{1+2/N}>0$ implies (\ref{S4C2E0}).
Thus, the second statement follows from the first statement.
\end{proof}

\section{Solvability in $\mathcal{L}^1_{\rm ul}(\RN)$}
Let $J(u)=u$, $\theta=1$ and $\xi=1$.
We use Theorem~\ref{TH1}~(i) to obtain the following:
\begin{corollary}\label{S5C1}
Let $u_0\ge 0$.
Suppose that $f\in C[0,\infty)$, and $f$ is nonnegative and nondecreasing.
If
\begin{equation}\label{S5C1E1}
\int_1^{\infty}\frac{\tf(u)du}{u^{1+2/N}}<\infty,\ \ \textrm{where}\ \ 
\tf(u):=\sup_{1\le\tau\le u}\frac{f(\tau)}{\tau},
\end{equation}
then (\ref{S1E1}) admits a local in time nonnegative solution $u(t)$, $0<t<T$, for each $u_0\in L^1_{\rm ul}(\RN)$.
Moreover, $\left\|u(t)\right\|_{L^1_{\rm ul}(\RN)}\le C$ for $0<t<T$.
\end{corollary}
As mentioned in Section~1, \cite{LRSV16} obtained a necessary and sufficient condition on $f$ for a solvability of (\ref{S1E1+}) in $L^1(\Omega)$.
Here, we use the following nonexistence result:
\begin{proposition}[{\cite[Theorem 4.1 and Lemma 4.2]{LRSV16}}]\label{S5P1}
Let $\Omega$ be a bounded domain in $\RN$.
Let $f\in C([0,\infty))$ be nonnegative and nondecreasing.
If
\begin{equation}\label{S5P1E1}
\int_1^{\infty}\frac{\tf(u)du}{u^{1+2/N}}=\infty,\ \ \textrm{where}\ \ \tf(u):=\sup_{1\le\tau\le u}\frac{f(\tau)}{\tau},
\end{equation}
then there is a nonnegative function $u_0\in L^1(\Omega)$ such that (\ref{S1E1+}) admits no local in time nonnegative solution in $L^1(\Omega)$.
Specifically, for each small $t>0$,
\begin{equation}\label{S5P1E2}
\int_{\Omega}\int_0^tS_{\Omega}(t-s)f(S_{\Omega}(s)u_0)dsdx=\infty.
\end{equation}
Here, $S_{\Omega}(t)[\phi](x):=\int_{\Omega}K_{\Omega}(x,y,t)\phi(y)dy$ and $K_{\Omega}(x,y,t)$ denotes the Dirichlet heat kernel on $\Omega$.
\end{proposition}
Using Corollary~\ref{S5C1} and Proposition~\ref{S5P1}, we obtain the following characterization:
\begin{theorem}\label{S5T1}
Let $u_0\ge 0$.
Suppose that $f\in C[0,\infty)$, $f$ is nonnegative and nondecreasing.
Then, (\ref{S1E1}) admits a local in time nonnegative solution for all $u_0\in\mathcal{L}^1_{\rm ul}(\RN)$ if and only if
(\ref{S5C1E1}) holds.
\end{theorem}

\begin{proof}
The sufficient part follows from Corollary~\ref{S5C1}.
Hereafter, we prove the necessary part.
Since $C_0^{\infty}(\RN)\subset L^1(\RN)$ is dense and $C_0^{\infty}(\RN)\subset BUC(\RN)$, we see that $L^1(\RN)\subset\mathcal{L}^1_{\rm ul}(\RN)$.
We assume (\ref{S5P1E1}).
Suppose the contrary, {\it i.e.}, (\ref{S1E1}) always has a nonnegative solution.
Let $u_0\in L^1(\Omega)$ be the initial function given in Proposition~\ref{S5P1}.
Here, we define $u_0=0$ in $\RN\setminus\Omega$.
Then, $u_0\in\mathcal{L}^1_{\rm ul}(\RN)$.
Since
\begin{equation}\label{S5T2E1}
u(t)=S(t)u_0+\int_0^tS(t-s)f(u(s))ds,
\end{equation}
we have
\begin{equation}\label{S5T2E2}
u(t)\ge S(t)u_0.
\end{equation}
By (\ref{S5T2E2}) and (\ref{S5T2E1}) we have
\[
u(t)\ge\int_0^tS(t-s)f(u(s))ds
\ge\int_0^tS(t-s)f(S(s)u_0)ds.
\]
Since two Green functions satisfy $K(x,y,t)\ge K_{\Omega}(x,y,t)$ (see, {\it e.g.}, \cite[Corollary 2.2]{B89}), we see that $S(t)u_0\ge S_{\Omega}(t)u_0$, and hence
\[
u(t)\ge\int_0^tS_{\Omega}(t-s)f(S_{\Omega}(s)u_0)ds.
\]
By (\ref{S5P1E2}) we see that $u(t)\not\in L^1_{\rm loc}(\RN)$ for small $t>0$, and hence (\ref{S1E1}) admits no nonnegative solution.
This is a contradiction.
Thus, the proof of the necessary part is complete, and the whole proof is also complete.
\end{proof}
Using Theorem~\ref{S5T1}, we show that the condition $\rho<1$ in Theorem~\ref{THB} is optimal.
Specifically, we cannot take $\rho=1$ in Theorem~\ref{THB}.

\begin{corollary}\label{S5C2}
Let $g(u):=u[\log(u+e)]^{\alpha}$, $\alpha\ge 0$, and let $f(u):=\frac{N}{2}g'(g^{-1}(u))g^{-1}(u)^{1+2/N}$.
Then the following hold:
\begin{enumerate}
\item $q=1+N/2$, $f\in X_q$ and
\begin{equation}\label{S5C2E1}
f'(u)F(u)-q=\frac{N\alpha}{2}\frac{h}{(h+e)\log(h+e)}
\frac{1+\frac{1}{h+e}+\frac{(\alpha-1)h}{(h+e)\log(h+e)}}{1+\frac{\alpha h}{(h+e)\log(h+e)}},
\end{equation}
where $h:=F(u)^{-N/2}$. In particular, (\ref{TH3E1}) with $\rho=1$ holds.
\item Let $J_{\alpha}(u):=g(F(u)^{-N/2})$, $\alpha:=N/2$.
Then, $J_{\alpha}(u)=u$ and there exists a nonnegative function $u_0\in\mathcal{L}^1_{\rm ul}(\RN)$ such that (\ref{S1E1}) admits no nonnegative solution.
\end{enumerate}
\end{corollary}

\begin{proof}
First, we prove (i).
By direct calculation we have $F(u)=(g^{-1}(u))^{-2/N}$ and $g(F(u)^{-N/2})=u$.
Differentiating $g(F(u)^{-N/2})=u$ with respect to $u$ twice, we obtain (\ref{S5C2E1}).
Since $h(u)(=g^{-1}(u))\to\infty$ $(u\to\infty)$, by (\ref{S5C2E1}) we see that $\lim_{u\to\infty}f'(u)F(u)=q$.
Since all other conditions on (\ref{f}) clearly hold, $f$ satisfies (\ref{f}), and hence $f\in X_q$ with $q=1+N/2$.
Using (\ref{S5C2E1}), a direct calculation reveals that (\ref{S1P2E1}) with $\rho=1$ holds.
The proof of (i) is complete.\\
\indent
Second, we prove (ii).
Let $u_1>1$ be large.
Since $J_{\alpha}(u)=u$, we have
\begin{multline*}
\int_1^{\infty}\frac{\tf(u)du}{u^{1+2/N}}
\ge\int_{u_1}^{\infty}\frac{f(u)du}{u^{2+2/N}}
=\int_{u_1}^{\infty}\frac{f(u)J_{\alpha}'(u)^2du}{J_{\alpha}(u)^{2+2/N}}
=\frac{N}{2}
\int_{F(u_1)^{-{N}/{2}}}^{\infty}\frac{g'(\tau)^2\tau^{1+2/N}d\tau}{g(\tau)^{2+2/N}}\\
=\int_{F(u_1)^{-{N}/{2}}}^{\infty}\frac{\left(1+\frac{N\tau}{2(\tau+e)\log(\tau+e)}\right)^2}{\tau\log(\tau+e)}d\tau
\ge\int_{F(u_1)^{-{N}/{2}}}^{\infty}\frac{d\tau}{(\tau+e)\log(\tau+e)}
=\left[\log\log(\tau+e)\right]_{F(u_1)^{-{N}/{2}}}^{\infty}=\infty,
\end{multline*}
where we used the change of variables $\tau=F(u)^{-N/2}$.
By Theorem~\ref{S5T1} we see that there exists a nonnegative function $u_0\in\mathcal{L}^1_{\rm ul}(\RN)$, which obviously implies $J_{\alpha}(u_0)\in\mathcal{L}^1_{\rm ul}(\RN)$,  such that (\ref{S1E1}) admits no nonnegative solution.
The proof of (ii) is complete.
\end{proof}
It follows from a direct calculation that $f$ given in Corollary~\ref{S5C2} behaves as follows:
\[
f(u)=\frac{N}{2}u^{1+{2}/{N}}[\log(u+e)]^{-2\alpha/N}(1+o(1))\ \ \textrm{as}\ \ u\to\infty.
\]

\section{The case $f_{\beta}(u)=u^{1+2/N}[\log(u+e)]^{\beta}$}
\begin{proof}[\bf Proof of Theorem~\ref{THD}]
We prove (i).
By \eqref{eq01+} with $\hat{\rho}=1$ and \eqref{eq02+} with $\hat{\rho}=1$ we see that
if there exists $\rho<1$ such that \eqref{TH3E1} holds, then $J_\alpha(u)$ is convex for large $u>0$.

It suffices to show that \eqref{TH3E1} holds in both cases (a) and (b).
By direct calculation we have
\begin{align*}
&(f'(u)F(u)-q)\log(h(u)+e)\\
&\hspace{10pt}=(f'(u)F(u)-q)\log(u+e)\cdot\frac{\log(h(u)+e)}{\log(u+e)}\\
&\hspace{10pt}=\left\{\frac{\beta u}{u+e}\cdot\frac{F(u)}{u^{-2/N}[\log(u+e)]^{-\beta}}
-\left(1+\frac{2}{N}\right)\left(\frac{N}{2}-\frac{F(u)}{u^{-2/N}[\log(u+e)]^{-\beta}}
\right)\log(u+e)
\right\}\cdot\frac{\log(h(u)+e)}{\log(u+e)}.
\end{align*}
By L'Hospital's rule we have
\begin{align*}
\lim_{u\to\infty} \frac{F(u)}{u^{-2/N}[\log(u+e)]^{-\beta}}&=
\lim_{u\to\infty} \frac{\frac{dF}{du}(u)}{\frac{d}{du}(u^{-2/N}[\log(u+e)]^{-\beta})}\\
&=\lim_{u\to\infty} \left(\frac{2}{N}+\frac{\beta u}{u+e}[\log(u+e)]^{-1}\right)^{-1}
=\frac{N}{2}.
\end{align*}
We see that
\[
\left(\frac{N}{2}-\frac{F(u)}{u^{-2/N}[\log(u+e)]^{-\beta}}\right)\log(u+e)
=\frac{\frac{N}{2}u^{-2/N}[\log(u+e)]^{-\beta}-F(u)}{u^{-2/N}[\log(u+e)]^{-\beta-1}},
\]
which implies that
\begin{multline*}
\lim_{u\to\infty} \left(\frac{N}{2}-\frac{F(u)}{u^{-2/N}[\log(u+e)]^{-\beta}}\right)\log(u+e)=
\lim_{u\to\infty} \frac{\frac{d}{du}\left(\frac{N}{2}u^{-2/N}[\log(u+e)]^{-\beta}-F(u)\right)}{\frac{d}{du}(u^{-2/N}[\log(u+e)]^{-\beta-1})}\\
={\lim_{u\to\infty}}\frac{-\frac{N}{2}\beta\cdot u^{-2/N}[\log(u+e)]^{-\beta-1}\cdot\frac{1}{u+e}}{-u^{-1-2/N}[\log(u+e)]^{-\beta-1}\left(\frac{2}{N}+(\beta+e)\cdot\frac{u}{u+e}[\log(u+e)]^{-1}\right)}
=\left(\frac{N}{2}\right)^2 \beta.
\end{multline*}
Moreover, since $\frac{d}{du}(u+e)/\{\frac{d}{du}(f(u)F(u))\}=1/(f'(u)F(u)-1)\to2/N$ as $u\to\infty$, we obtain
\[
\frac{\frac{d}{du}\left\{\log(h(u)+e)\right\}}{\frac{d}{du}\left\{\log(u+e)\right\}}
=\frac{\frac{N}{2}\cdot\frac{u+e}{f(u)F(u)}}{1+F(u)^{N/2}}\to1 \ \ \text{as $u\to\infty$,}
\]
and hence ${\log(h(u)+e)}/{\log(u+e)}\to1$ as $u\to\infty$.
Therefore, we have
\begin{eqnarray}
\label{eq40}
\lim_{u\to\infty} (f'(u)F(u)-q)\log(h(u)+e)=-\left(\frac{N}{2}\right)^2 \beta.
\end{eqnarray}
Since $\beta\ge-1$ and $\alpha>N/2$, or $\beta>-1$ and $\alpha=N/2$,
we have $-{N\beta}/{(2\alpha)}<1$. Then we can choose 
$\rho$ such that $-{N\beta}/{(2\alpha)}<\rho<1$, which leads to $-\left(\frac{N}{2}\right)^2 \beta<\frac{N}{2}\alpha\rho$. By this together with \eqref{eq40} we obtain
\eqref{TH3E1} in both cases (a) and (b).

We prove (c).
Let $-(1+2/N)\kappa<\beta<-1$, $\theta=1$ and $J(u)=u$.
In order to apply Theorem~\ref{TH1}~(i) we check all the assumptions.
Since $\frac{d}{du}\left(\frac{J'(u)}{J(u)^{1-\theta}}\right)=0$ and
\[
\frac{d}{du}\left(\frac{f(u)}{J(u)^{\theta}}\right)
=\frac{2}{N}u^{\frac{2}{N}-1}[\log(u+e)]^{\beta}
\left(1+\frac{\beta u}{(u+e)\log(u+e)}\right)>0\ \ \textrm{for large}\ u>0,
\]
$f(u)/J(u)^{\theta}$ and $J'(u)/J(u)^{1-\theta}$ are nondecreasing for large $u>0$.
If $\eta>0$ is large, then
\begin{multline*}
\tJ(\eta)\int_{\eta}^{\infty}\frac{\tf(\tau)J'(\tau)d\tau}{J(\tau)^{1+2/N}}=
\frac{J'(\eta)}{J(\eta)^{1-\theta}}
\int_{\eta}^{\infty}\frac{f(\tau)J'(\tau)d\tau}{J(\tau)^{1+\theta+2/N}}
=\int_{\eta}^{\infty}\frac{[\log(\tau+e)]^{\beta}d\tau}{\tau}\\
\le\int_{\eta}^{\infty}\frac{2[\log(\tau+e)]^{\beta}d\tau}{\tau+e}
=\frac{2}{-\beta-1}[\log(\eta+e)]^{\beta+1}\to 0\ \ \textrm{as}\ \ \eta\to\infty.
\end{multline*}
Then, (\ref{TH1E1}) holds.
It follows from Theorem~\ref{TH1}~(i) that (\ref{S1E1}) has a nonnegative solution.
A proof of (c) is complete.

We prove (a) of (ii).
It suffices to show that $f=f_\beta$ satisfies all the assumptions of Theorem~\ref{THC}.
We easily see that $f$ satisfies (\ref{f}).
We check the assumption Theorem~\ref{THC}~(i).
When $\beta>0$, since $F^{-1}_{\beta}\circ F_{\beta}(u)=u$, Theorem~\ref{THC}~(i) holds.
When $-1<\beta\le0$, we have $f'(u)F(u)\ge1+N/2$ for large $u>0$.
Thus, Theorem~\ref{THC}~(i) follows from Remark~\ref{S1R3}.
We check the assumption Theorem~\ref{THC}~(ii).
We choose $\delta$ such that $\max\{-\beta,0\}<\delta<1$.
Since $\beta+\delta>0$, we have
\[
\lim_{u\to\infty} \frac{\frac{dF}{du}(u)}{\frac{d}{du}\left\{u^{-2/N}[\log(u+e)]^{\delta}\right\}}
=\lim_{u\to\infty} \frac{1}{[\log(u+e)]^{\beta+\delta}\left(\frac{2}{N}-\frac{\delta u}{(u+e)\log(u+e)}\right)}=0.
\]
By L'Hospital's rule we have
$F(u)/\left\{u^{-2/N}[\log(u+e)]^{\delta}\right\}\to0$ as $u\to\infty$. 
Thus, Theorem~\ref{THC}~(ii) also holds.
Applying Theorem~\ref{THC}, we obtain (a) of (ii).

We prove (b) of (ii).
Since $\beta\ge -(1+2/N)\kappa$, $f_\beta(u)$ is nondecreasing, and Theorem~\ref{S5T1} is applicable.
Since there is $\sigma>1$ such that $\tf(\tau)=f(\tau)/\tau$ for $\tau\ge \sigma$,  we have
\[
\int_1^{\infty}\frac{\tf(\tau)d\tau}{\tau^{1+2/N}}
\ge\int_{\sigma}^{\infty}\frac{f(\tau)d\tau}{\tau^{2+2/N}}
\ge\int_{\sigma}^{\infty}\frac{d\tau}{(\tau+e)\log(\tau+e)}
=\left[\log\log(\tau+e)\right]_{\sigma}^{\infty}
=\infty,
\]
and hence it follows from Theorem~\ref{S5T1} that there exists a nonnegative function $u_0\in\mathcal{L}^1_{\rm ul}(\RN)$ such that (\ref{S1E1}) admits no nonnegative solution.
Next, we show that $J_{\alpha}(u_0)\in\mathcal{L}^1_{\rm ul}(\RN)$.
We have
\[
F(u)\ge\log(u+e)\int_u^{\infty}\frac{d\tau}{\tau^{1+2/N}}=\frac{N}{2}\frac{\log(u+e)}{u^{2/N}}.
\]
Hence,
\begin{equation}\label{THDPE1}
h(u):=F(u)^{-N/2}\le C\frac{u}{[\log(u+e)]^{N/2}}.
\end{equation}
Since $0\le \alpha\le N/2$, by (\ref{THDPE1}) we  see that
\begin{equation}\label{THDPE2}
\frac{[\log(h(u)+e)]^{\alpha}}{[\log(u+e)]^{N/2}}\le C\ \ \textrm{for}\ \ 
u\ge 0.
\end{equation}
By (\ref{THDPE1}) and (\ref{THDPE2}) we have
\begin{multline}\label{THDPE3}
0\le J'_{\alpha}(u)
=\frac{N}{2}\left(1+\frac{\alpha h(u)}{(h(u)+e)\log(h(u)+e)}\right)\frac{[\log(h(u)+e)]^{\alpha}}{f(u)F(u)^{1+N/2}}\\
\le C\frac{[\log(h(u)+e)]^{\alpha}}{u^{1+2/N}[\log (u+e)]^{-1}}\left(\frac{Cu}{[\log(u+e)]^{N/2}}\right)^{1+2/N}
\le C\frac{[\log(h(u)+e)]^{\alpha}}{[\log(u+e)]^{N/2}}\le C\ \ \textrm{for}\ \ 
u\ge 0.
\end{multline}
Since $u_0\in\mathcal{L}^1_{\rm ul}(\RN)$, by (\ref{THDPE3}) we see that $J_{\alpha}(u_0)\in\mathcal{L}^1_{\rm ul}(\RN)$.

Proofs of all the cases are complete.
\end{proof}

\section{Regularly varying functions and rapidly varying functions}
In this section we always assume the two exponents $p$ and $q$ always satisfy
\[
\begin{cases}
\frac{1}{p}+\frac{1}{q}=1 & \textrm{if}\ 1<q<\infty,\\
p=\infty & \textrm{if}\ q=1.
\end{cases}
\]

The following theorem is a fundamental property of ${\rm RV}_p$:
\begin{proposition}[Karamata's representation theorem]\label{S7P1}
There exsit functions $a(s)$ and $b(u)$ with
\[
\lim_{u\to\infty}b(u)=b_0\ (0<b_0<\infty)
\ \ \textrm{and}\ \ 
\lim_{s\to\infty}a(s)=p\ (0\le p<\infty)
\]
and $u_0\ge 0$ such that for $u>u_0$,
\begin{equation}\label{S7P1E1}
f(u)=b(u)\exp\left(\int_{u_0}^u\frac{a(s)}{s}ds\right)
\end{equation}
if and only if $f\in {\rm RV}_p$ $(0\le p<\infty)$.
\end{proposition}
See \cite[Theorem~1.5]{GD87} for details.
Note that in this section $u_0$ does not stand for an initial function.

Hereafter, we assume that $f$ satisfies (\ref{f}).
\begin{lemma}\label{S7L1}
Suppose that $f\in X_q$ for some $q\in(1,\infty)$.
Then there exist $a(s)$ and $b(u)$ with
\[
\lim_{u\to\infty}b(u)=b_0\ (0<b_0<\infty)\ \ \textrm{and}\ \ 
\lim_{s\to\infty}a(s)=p
\]
and $u_0>0$ such that (\ref{S7P1E1}) holds for $u>u_0$.
\end{lemma}
\begin{proof}
Let $\eta(u)\in C[0,\infty)$ such that $f'(u)F(u)=q+\eta (u)$.
Then $\eta(u)\to\ 0$ as $u\to\infty$.
We have $(f(u)F(u))'=q-1+\eta(u)$.
Integrating it over $[u_0,u]$, we have
\begin{equation}\label{S7L1E1}
f(u)F(u)=(q-1)u+h(u),
\end{equation}
where $h(u):=\int_{u_0}^u\eta(s)ds+f(u_0)F(u_0)-(q-1)u_0$.
Integrating $\frac{1}{f(u)F(u)}=\frac{1}{(q-1)u+h(u)}$ over $[u_0,u]$, we have $-\log\frac{F(u)}{F(u_0)}=\int_{u_0}^u\frac{ds}{(q-1)s+h(s)}$.
Hence,
\begin{equation}\label{S7L1E2}
\frac{1}{F(u)}=\frac{1}{F(u_0)}\exp\left(\int_{u_0}^u\frac{ds}{(q-1)s+h(s)}\right).
\end{equation}
By (\ref{S7L1E1}) and (\ref{S7L1E2}) we have
\[
f(u)=\frac{(q-1)u+h(u)}{F(u_0)}\exp\left(\int_{u_0}^u\frac{ds}{(q-1)s+h(s)}\right).
\]
Thus, we obtain (\ref{S7P1E1}), where
\[
a(u):=\frac{p+\rho (u)}{1+\rho (u)},\ \ 
b(u):=\frac{(q-1)u_0}{F(u_0)}(1+\rho(u)),\ \ 
 p:=\frac{q}{q-1}
\ \ \textrm{and}\ \ \rho(u):=\frac{h(u)}{(q-1)u}.
\]
Since $\rho(u)\to 0$ $(u\to\infty)$, we see that $b(u)\to b_0>0$ $(u\to\infty)$ and $a(u)\to p$ $(u\to\infty)$.
The proof is complete.
\end{proof}

\begin{proof}[\bf Proof of Theorem~\ref{THE}~(i)]
We consider the case $1<q<\infty$.
Let $f\in X_q$.
It follows from Lemma~\ref{S7L1} that there exist $a(s)$ and $b(u)$ such that $f(u)=b(u)\exp\left(\int_{u_0}^ua(s)ds/s\right)$, where $b(u)\to b_0>0$ $(u\to\infty)$ and $a(u)\to p:=q/(q-1)$ $(u\to\infty)$.
By Proposition~\ref{S7P1} we see that $f\in {\rm RV}_p$.

We consider the case $q=1$.
Let $f\in X_1$.
Since $f'(u)F(u)\to 1$ $(u\to\infty)$, for any $\e>0$, there is $u_{\e}>0$ such that $|f'(u)F(u)-1|<\e$ for $u>u_{\e}$.
By the mean value theorem we see that $0\le f(u)F(u)\le f(u_{\e})F(u_{\e})+\e(u-u_{\e})$ for $u>u_{\e}$.
We have
\[
\frac{u}{f(u)F(u)}\ge\frac{u}{f(u_{\e})F(u_{\e})+\e(u-u_{\e})}\to\frac{1}{\e}\ \ \textrm{as}\ \ u\to\infty.
\]
Since $\e>0$ can be chosen arbitrary small, we see that $\lim_{u\to\infty}u/f(u)F(u)=\infty$.
Then,
\begin{equation}\label{THEPE1}
\lim_{u\to\infty}\frac{uf'(u)}{f(u)}=\lim_{u\to\infty}f'(u)F(u)\frac{u}{f(u)F(u)}=\infty.
\end{equation}
Let $a(u):=uf'(u)/f(u)$.
Then, we easily see that $f(u)=f(u_0)\exp\left(\int_{u_0}^ua(s)ds/s\right)$.
It follows from (\ref{THEPE1}) that for any $M>0$, there is $u_M>0$ such that $a(u)>M$ for $u>u_M$.
Let $\lambda>1$.
For $u>u_M$,
\[
\frac{f(\lambda u)}{f(u)}=\exp\left(\int_u^{\lambda u}\frac{a(s)}{s}ds\right)
\ge\exp\left(\int_u^{\lambda u}\frac{M}{s}ds\right)=\lambda^M.
\]
Since $M$ is arbitrary large, we see that $f(\lambda u)/f(u)\to\infty$ as $u\to\infty$.
When $0<\lambda<1$, by similar way we can show that $f(\lambda u)/f(u)\to 0$ as $u\to\infty$.
Thus, $f\in{\rm RV}_{\infty}$.
\end{proof}

\begin{lemma}\label{S7L2}
The following hold:\\
(i) Suppose that $f'$ is nondecreasing.
If $f\in{\rm RV}_p$ for some $p\in(1,\infty)$, then $f\in X_q$.\\
(ii) Suppose that $f'(u)F(u)$ is nondecreasing.
If $f\in{\rm RV}_{\infty}$, then $f\in X_1$.
\end{lemma}
\begin{proof}
(i)  We follow the strategy used in \cite[Proposition~1.7.11]{GD87}.
Let $1<p<\infty$ and $f\in {\rm RV}_p$.
Let $\lambda>1$.
Since $f'$ is nondecreasing,
\[
\frac{u(\lambda-1)f'(u)}{f(u)}\le\int_1^{\lambda}\frac{uf'(\mu u)}{f(u)}d\mu=\frac{f(\lambda u)-f(u)}{f(u)}.
\]
Since $f\in {\rm RV}_p$, we see that $\limsup_{u\to\infty}\frac{uf'(u)}{f(u)}\le\frac{\lambda^p-1}{\lambda-1}$ for all $\lambda>1$.
Letting $\lambda\downarrow 1$, we have $\limsup_{u\to\infty}\frac{uf'(u)}{f(u)}\le p$.
Let $0<\lambda<1$. Since
\begin{equation}\label{S7L2E1}
\frac{u(1-\lambda)f'(u)}{f(u)}\ge\int_{\lambda}^1\frac{uf'(\mu u)}{f(u)}d\mu=\frac{f(u)-f(\lambda u)}{f(u)},
\end{equation}
we have $\liminf_{u\to\infty}\frac{uf'(u)}{f(u)}\ge\frac{\lambda^p-1}{\lambda-1}$.
Letting $\lambda\uparrow 1$, we have $\liminf_{u\to\infty}\frac{uf'(u)}{f(u)}\ge p$.
Thus, $\lim_{u\to\infty}\frac{uf'(u)}{f(u)}=p$.
Since $p>1$, we can show that $\lim_{u\to\infty}\frac{u}{f(u)}=0$.
By L'Hospital's rule we have
\[
\lim_{u\to\infty}\frac{F(u)}{\frac{u}{f(u)}}
=\lim_{u\to\infty}\frac{-\frac{1}{f(u)}}{\frac{1}{f(u)}-\frac{uf'(u)}{f(u)^2}}=\frac{1}{p-1}.
\]
Then,
\[
\lim_{u\to\infty}f'(u)F(u)=\lim_{u\to\infty}\frac{uf'(u)}{f(u)}\frac{F(u)}{\frac{u}{f(u)}}=\frac{p}{p-1},
\]
and hence $f\in X_q$.\\
(ii) Let $f\in{\rm RV}_{\infty}$.
Since $F(u)$ is decreasing and $f'(u)F(u)$ is nondecreasing, $f'(u)$ is nondecreasing.
Let $0<\lambda<1$. Then, (\ref{S7L2E1}) holds.
Since $f\in {\rm RV}_{\infty}$ and $0<\lambda<1$, we see that $\lim_{u\to\infty}f(\lambda u)/f(u)=0$.
Then,
\[
\frac{uf'(u)}{f(u)}\ge\frac{1}{1-\lambda}\left(1-\frac{f(\lambda u)}{f(u)}\right)\to\frac{1}{1-\lambda}\ \ \textrm{as}\ \ u\to\infty.
\]
Letting $\lambda\uparrow 1$, we have $\lim_{u\to\infty}uf'(u)/f(u)=\infty$.
By L'Hospital's rule we have
\begin{equation}\label{S7L2E2}
\lim_{u\to\infty}\frac{f(u)F(u)}{u}=\lim_{u\to\infty}\frac{F(u)}{\frac{u}{f(u)}}
=\lim_{u\to\infty}\frac{-\frac{1}{f(u)}}{\frac{1}{f(u)}-\frac{uf'(u)}{f(u)^2}}=0.
\end{equation}
Let $\lambda>1$.
Since $f'(u)F(u)$ is nondecreasing, by (\ref{S7L2E2}) we have
\begin{multline*}
(\lambda-1)(f'(u)F(u)-1)\le\int_1^{\lambda}(f'(\mu u)F(\mu u)-1)d\mu
=\int_1^{\lambda}\frac{1}{u}\frac{d}{d\mu}\left(f(\mu u)F(\mu u)\right)d\mu\\
=\frac{f(\lambda u)F(\lambda u)}{\lambda u}\lambda-\frac{f(u)F(u)}{u}\to 0\ \ \textrm{as}\ \ u\to\infty,
\end{multline*}
and hence $\limsup_{u\to\infty}(f'(u)F(u)-1)\le 0$.
Let $0<\lambda <1$. Then
\[
(1-\lambda)(f'(u)F(u)-1)\ge\int_{\lambda}^1\left(f'(\mu u)F(\mu u)-1\right)d\mu
=\frac{f(u)F(u)}{u}-\frac{f(\lambda u)F(\lambda u)}{\lambda u}\lambda\to 0\ \ \textrm{as}\ \ u\to\infty,
\]
and hence $0\le\liminf_{u\to\infty}(f'(u)F(u)-1)$.
Thus, 
\[
0\le\liminf_{u\to\infty}(f'(u)F(u)-1)\le \limsup_{u\to\infty}(f'(u)F(u)-1)\le 0,
\]
and hence $\lim_{u\to\infty}f'(u)F(u)=1$.
We see that $f\in X_1$.
\end{proof}

\begin{proof}[\bf Proof of Theorem~\ref{THE}~(ii) and (iii)]
Theorem~\ref{THE}~(ii) and (iii) follow from Theorem~\ref{THE}~(i) and Lemma~\ref{S7L2}.
\end{proof}

\section{Summary and problems}
In this paper we study integrability conditions on $u_0$ which determines existence and nonexistence of a local in time nonnegative solution of (\ref{S1E1}).
In a critical and subcritical cases existence and nonexistence integrability conditions on $u_0$ are given by Theorem~\ref{THA} and Proposition~\ref{S1P2}~(ii).
In the doubly critical case these conditions are given by Theorems~\ref{THB} and \ref{THC}.
See Figure~\ref{fig1}.
When $f(u)=u^{1+2/N}[\log (u+e)]^{\beta}$, $\beta\ge-(1+2/N)\kappa$, where $\kappa$ is given by (\ref{kappa}),
the problem becomes a doubly critical case and a complete classification is given by Theorem~\ref{THD}.
Theorems~\ref{THA} and \ref{THB} can be applied to a nonlinearity in $X_q$.
A characterization of $X_q$ is given in Theorem~\ref{THE}.

We also study Problem (B) stated in Section~1.
Corollaries~\ref{S3C1} and \ref{S3C2} are sufficient conditions on $f$ for existence when $J$ is given.
Corollaries~\ref{S4C1} and \ref{S4C2} are sufficient conditions on $f$ for nonexistence when $J$ is given.
In Sections~5 and 9 we give a necessary and sufficient condition on $f$ for an existence of a nonnegative solution of (\ref{S1E1}) for every nonnegative function $u_0\in\mathcal{L}^r_{\rm ul}(\RN)$.
Section 5 (resp.\ 9) is for the case $r=1$ (resp. $r>1$).
This necessary and sufficient condition corresponds to \cite[Corollary~4.5 and Theorem~3.4]{LRSV16} which studied in the $L^r(\Omega)$ framework.

An objective of this study is to prove Table~\ref{tab1} under mild assumptions on $f$ and $J$.
This problem derives several concrete problems.

In the proof of Theorems~\ref{THA} and \ref{THB} we use Theorem~\ref{TH1} which relates the nonlinearity $f$ and the integrability $J(u_0)\in L^1_{\rm ul}(\RN)$.
The condition (\ref{TH1E2}) is a sufficient condition for an existence.
Since there is a gap between (\ref{TH1E2}) and (\ref{S4C1E0}) (or between (\ref{TH1E2}) and (\ref{S4C2E0})), it is natural to ask the following:
\begin{problem}
Suppose that (\ref{TH1E2}) does not hold.
Does there exist $u_0\ge 0$ such that $J(u_0)\in\mathcal{L}^1_{\rm ul}(\RN)$ and (\ref{S1E1}) admits no nonnegative solution?
\end{problem}

Corollaries~\ref{S3C1}, \ref{S3C2}, \ref{S4C1} and \ref{S4C2} are partial answers to Problem (B) and they are not optimal.
We do not know whether (\ref{S3C1E-1}), (\ref{S3C2E-1}), (\ref{S4C1E0-}) and (\ref{S4C2E-1}) are technical conditions or not.
\begin{problem}
Can one prove a theorem similar to Corollary~\ref{S4C1} (resp. Corollary~\ref{S4C2}) without assuming (\ref{S4C1E0-}) (resp. (\ref{S4C2E-1}))?
\end{problem}
There is a gap between (\ref{S3C1E1}) and (\ref{S4C1E0}) (or between (\ref{S3C2E0}) and (\ref{S4C2E0})).
\begin{problem}
Can one obtain a growth condition on $f$ which is shaper than (\ref{S3C1E1}), (\ref{S3C2E0}), (\ref{S4C1E0}) or (\ref{S4C2E0})?
\end{problem}
Theorem~\ref{THC} is a sufficient condition for nonexistence in a doubly critical case.
The assumption Theorem~\ref{THC}~(i) and (ii) seem technical.
\begin{problem}
Can one obtain a nonexistence result for a wide class of nonlinearities in a doubly critical case?
\end{problem}

\section{Appendix to \cite{LRSV16}: Solvability in $\mathcal{L}^r_{\rm ul}(\RN)$}
We recover \cite[Theorem 3.4 and Corollary 4.5]{LRSV16} in a framework of uniformly local Lebesgue spaces.
Only in this section we adopt the following definition of a solution:
\begin{definition}\label{S9D1}
Let $r\ge 1$.
We call $u(t)$ a solution of (\ref{S1E1}) if $u(t)$ is a solution in the sense of Definition~\ref{S1D1} and $u(t)\in L^{\infty}((0,T),L^r_{\rm ul}(\RN))$.
\end{definition}

In Theorem~\ref{S5T1} we already obtained a necessary and sufficient condition for an existence of a nonnegative solution of (\ref{S1E1}) in $\mathcal{L}^1_{\rm ul}(\RN)$ in the sense of Definition~\ref{S1D1}.
Since the solution $u(t)$ satisfies $u(t)\in L^{\infty}((0,T),L^1_{\rm ul}(\RN))$, $u(t)$ is also a solution in the sense of Definition~\ref{S9D1} with $r=1$.
\begin{corollary}\label{S9C1}
Suppose that $f\in C[0,\infty)$, $f$ is nonnegative and nondecreasing.
Then (\ref{S1E1}) admits a local in time nonnegative solution in the sense of Definition~\ref{S9D1} for every nonnegative function $u_0\in\mathcal{L}^1_{\rm ul}(\RN)$ if and only if (\ref{S5C1E1}) holds.
\end{corollary}

Hereafter, we consider the case $r>1$.
\begin{theorem}\label{S9T3}
Let $r>1$.
Suppose that $f\in C[0,\infty)$, $f$ is nonnegative and nondecreasing.
Then, (\ref{S1E1}) admits a local in time nonnegative solution in the sense of Definition~\ref{S9D1} for every nonnegative function $u_0\in\mathcal{L}^r_{\rm ul}(\RN)$ if and only if
\begin{equation}\label{S9T3E1}
\limsup_{u\to\infty}\frac{f(u)}{u^{1+2r/N}}<\infty.
\end{equation}
\end{theorem}
When the function space is $L^r(\Omega)$, Theorem~\ref{S9T3} was obtained in \cite{LRSV16}.
Theorem~\ref{S9T3} corresponds to \cite[Theorem~3.4]{LRSV16}.
Since we work in $L^r_{\rm ul}(\RN)$, we do not have to care about a behavior of $f(u)$ near $u=0$.

\begin{proof}
First, we prove the sufficient part.
Specifically, we prove the existence of a solution provided that (\ref{S9T3E1}) holds.
Let $J(u):=u^r$.
We consider the case $r\ge N/(N-2)$ and $N\ge 3$.
Let $\theta:=1/r+2/N$.
Then, $r\theta=1+2r/N$ and $0<\theta\le 1$.
Since there is $C>0$ such that
\[
\tf(u)=\sup_{1\le\tau\le u}\frac{f(\tau)}{\tau^{r\theta}}<C,
\]
we have
\[
\tJ(\eta)\int_{\eta}^{\infty}\frac{\tf(\tau)J'(\tau)d\tau}{J(\tau)^{1+2/N}}
\le r^2\eta^{2/N}\int_{\eta}^{\infty}\frac{\tf(\tau)d\tau}{\tau^{1+2/N}}
\le r^2\eta^{2/N}\frac{N}{2}\frac{C}{\eta^{2/N}}=\frac{CNr^2}{2}.
\]
Since $u_0\in\mathcal{L}^r_{\rm ul}(\RN)$ implies $u_0^r\in\mathcal{L}^1_{\rm ul}(\RN)$, it follows from Theorem~\ref{TH1}~(ii) that (\ref{S1E1}) has a nonnegative solution $u(t)$ in the sense of Definition~\ref{S1D1} and $\left\|u(t)^r\right\|_{L^1_{\rm ul}(\RN)}<C_1$ for small $t>0$.
Since $\left\|u(t)\right\|_{L^r_{\rm ul}(\RN)}<C_1$ for small $t>0$, $u(t)$ is a nonnegative solution in the sense of Definition~\ref{S9D1}.
We consider the case $N=1,2$ or $1<r<N/(N-2)$ and $N\ge 3$.
Let $\theta=1$.
Then, $1/r+2/N>1$.
Since
\[
\tf(\tau)\le
\sup_{1\le\sigma\le\tau}\left(\frac{f(\sigma)}{J(\sigma)^{\theta}}\frac{J(\tau)^{1/r+2/N-1}}{J(\sigma)^{1/r+2/N-1}}\right)
=\sup_{1\le\sigma\le\tau}\left(\frac{f(\sigma)}{J(\sigma)^{1/r+2/N}}\right)J(\tau)^{1/r+2/N-1}
\le C\tau^{1+2r/N-r},
\]
we have
\[
\tJ(\eta)\int_{\eta}^{\infty}\frac{\tf(\tau)J'(\tau)d\tau}{J(\tau)^{1+2/N}}
\le r\eta^{r-1}\int_{\eta}^{\infty}\frac{Cr\tau^{r-1}}{\tau^{2r-1}}
=r^2\eta^{r-1}\frac{C}{(r-1)\eta^{r-1}}=\frac{Cr^2}{r-1}.
\]
Since $u_0^r\in\mathcal{L}^1_{\rm ul}(\RN)$, it follows from Theorem~\ref{TH1}~(ii) and the same argument above that (\ref{S1E1}) has a nonnegative solution in the sense of Definition~\ref{S9D1}.
The proof of the sufficient part is complete.

Second, we prove the necessary part.
Specifically, we prove that for a certain nonnegative function $u_0\in\mathcal{L}^r_{\rm ul}(\RN)$, (\ref{S1E1}) admits no nonnegative solution provided that (\ref{S9T3E1}) does not hold.
Let $\Omega\subset\RN$ be a bounded domain.
In \cite[Theorem~3.3]{LRSV16} it was shown that if (\ref{S9T3E1}) does not hold, {\it i.e.},
\[
\limsup_{u\to\infty}\frac{f(u)}{u^{1+2r/N}}=\infty,
\]
then there exists a nonnegative function $u_0\in L^r(\Omega)$ such that
\[
\int_{\Omega}\left| \int_0^tS_{\Omega}(t-s)f(S_{\Omega}(s)u_0)ds\right|^rdx=\infty\ \ \textrm{for small}\ t>0.
\]
We can easily see that $u_0\in\mathcal{L}^r(\RN)$, because $u_0\in L^r(\RN)$ and $C_0^{\infty}(\RN)$ is dense in $L^r(\RN)$.
If a solution $u(t)$ of (\ref{S1E1}) exists, then by the same argument as in the proof of Theorem~\ref{S5T1} we see that
\[
\left\|u(t)\right\|^r_{L^r(\Omega)}\ge
\int_{\Omega}\left| \int_0^tS_{\Omega}(t-s)f(S_{\Omega}(s)u_0)ds\right|^rdx=\infty\ \ \textrm{for small}\ t>0,
\]
which indicates that (\ref{S1E1}) with $u_0$ admits no nonnegative solution in the sense of Definition~\ref{S9D1}.
The proof is complete.
\end{proof}

\section{Appendix to \cite{GM21}}
In the proof of \cite[Theorems 1.4 (ii) and 1.6 (ii)]{GM21} the following was claimed:
Let $0<r<N/2$ and $2<\alpha<N/r$.
The function
\[
u_0(x):=
\begin{cases}
F^{-1}(|x|^{\alpha}) & \textrm{if}\ F(0)=\infty,\\
F^{-1}(\min\{ |x|^{\alpha},F(0)\}) & \textrm{if}\ F(0)<\infty
\end{cases}
\]
satisfies $F(u_0)^{-r}\in L^1_{\rm ul}(\RN)$ and (\ref{S1E1}) admits no nonnegative solution.

As mentioned in Remark~\ref{S1R1}~(vi), $F(u_0)^{-r}\in L^1_{\rm ul}(\RN)$ does not necessarily imply $u_0\in L^1_{\rm ul}(\RN)$ if $q>1+r$.
In that case it may occur that $S(t)u_0\not\in L^{\infty}((0,T),L^1_{\rm ul}(\RN))\cap L^{\infty}_{\rm loc}((0,T),L^{\infty}(\RN))$, and hence (\ref{S1D1E1}) does not hold.
Thus, the assumption $q\le 1+r$ should be added in \cite[Theorems 1.4 (ii) and 1.6 (ii)]{GM21} as follows:
\begin{proposition}
Let $f\in X_q$.
If $q\le 1+r$, then for each $r\in (0,N/2)$, there is a nonnegative function $u_0\in L^1_{\rm ul}(\RN)$ such that $F(u_0)^{-r}\in L^1_{\rm ul}(\RN)$ and (\ref{S1E1}) admits no nonnegative solution.
\end{proposition}

\begin{proof}
If we prove
\begin{equation}\label{S8E0}
u_0\in L^1_{\rm ul}(\RN),
\end{equation}
then by Proposition~\ref{S2P2} we see that $S(t)u_0\in L^{\infty}((0,T),L^1_{\rm ul}(\RN))\cap L^{\infty}_{\rm loc}((0,T),L^{\infty}(\RN))$.

Hereafter, we prove (\ref{S8E0}).
Since $f'(u)F(u)\to q$ for large $u$, there is $\e_1>0$ such that
\[
f'(u)F(u)\le q+\e_1
\]
for large $u$.
Integrating $f'(u)/f(u)\le (q+\e_1)/f(u)F(u)$ over $[u_0,u]$, we have $f(u)F(u)^{q+\e_1}\le C_1$.
Integrating $1/C_1\le 1/f(u)F(u)^{q+\e_1}$ over $[u_0,u]$, we have $F(u)\le C_3(u-C_2)^{-1/(q-1+\e_1)}$.
Therefore, there is a large $C_4>0$ such that $F(u)\le C_4 u^{-1/(q-1+\e_1)}$ for large $u$.
Since $F^{-1}$ is decreasing, $u\ge F^{-1}(C_4u^{-1/(q-1+\e_1)})$, and hence $F^{-1}(u)\le C_5u^{-(q-1+\e_1)}$.
Let $\rho>0$ be small.
Then,
\begin{equation}\label{S8E1}
\int_{|x|\le\rho}F^{-1}(|x|^{\alpha})dx
\le C_6\int_{0}^{\rho}r^{-\alpha(q-1+\e_1)+N-1}dr.
\end{equation}
Since $q-1\le r$ and $\alpha=N(1-\e_2)/r$ for some $\e_2>0$, we have
\[
-\alpha(q-1+\e_1)+N-1\ge -\frac{N}{r}(1-\e_2)(r+\e_1)+N-1
=-1+\e_2N-\e_1(1-\e_2)\frac{N}{r}.
\]
Since $\e_1>0$ can be taken arbitrary small, we can take $\e_1>0$ such that $-\alpha(q-1+\e_1)+N-1>-1$.
By (\ref{S8E1}) we see that $\int_{|x|\le\rho}F^{-1}(|x|^{\alpha})dx<\infty$ and it indicates (\ref{S8E0}).
By the proof of \cite[Theorems 1.4 (ii) and 1.6 (ii)]{GM21} we see that the conclusion of the proposition holds.
\end{proof}


\end{document}